\documentclass[a4paper,11pt]{amsart}
\usepackage[all]{xy}
\usepackage{latexsym}
\usepackage{amssymb}
\usepackage{amsmath}
\usepackage{mathrsfs}
\usepackage{hyperref}
\usepackage{color}

\numberwithin{equation}{section} \theoremstyle{plain}
\newtheorem*{thm*}{Main Theorem}
\newtheorem{thm}{Theorem}
\newtheorem{coro}[thm]{Corollary}
\newtheorem*{coro*}{Corollary}
\newtheorem{lem}[thm]{Lemma}
\newtheorem*{lem*}{Lemma}
\newtheorem{prop}[thm]{Proposition}
\newtheorem*{prop*}{Proposition}
\newtheorem{rem}[thm]{Remark}
\newtheorem*{rem*}{Remark}
\newtheorem{exa}[thm]{Example}
\newtheorem*{exa*}{Example}
\newtheorem{df}[thm]{Definition}
\newtheorem*{df*}{Definition}
\newtheorem*{ack*}{ACKNOWLEDGEMENTS}

\newcommand{\Ext}{\mbox{\rm Ext}}
\newcommand{\extri}{\mbox{\rm extri}}

\newcommand{\Hom}{\mbox{\rm Hom}}
\newcommand{\Proj}{\mbox{\rm Proj}}
\newcommand{\proj}{\mbox{\rm proj}}
\newcommand{\Inj}{\mbox{\rm Inj}}
\newcommand{\inj}{\mbox{\rm inj}}
\newcommand{\Ab}{\mbox{\rm Ab}}
\newcommand{\Ob}{\mbox{\rm Ob}}

\newcommand{\Id}{\mbox{\rm Id}}

\newcommand{\Jac}{\mbox{\rm Jac}}
\newcommand{\Mod}{\mbox{\rm Mod}}

\newcommand{\End}{\mbox{\rm End}}
\newcommand{\im}{\mbox{\rm Im\,}}

\newcommand{\Ker}{\mbox{\rm Ker\,}}
\newcommand{\WE}{\rm WE}

\newcommand{\normal}{\lhd}
\newcommand{\projR}{{\rm proj}\mbox{-}R}
\newcommand{\Rproj}{R\mbox{-}{\rm proj}}
\newcommand{\RMod}{R\mbox{-}{\rm Mod}}
\newcommand{\Rmod}{R\mbox{-}{\rm mod}}

\newcommand{\mcA}{\mathcal A}

\newcommand{\mcE}{\mathcal E}
\newcommand{\mcI}{\mathcal I}
\newcommand{\mcJ}{\mathcal J}
\newcommand{\mcM}{\mathcal M}
\newcommand{\mcS}{\mathcal S}
\newcommand{\mcT}{\mathcal T}
\newcommand{\mcF}{\mathcal F}
\newcommand{\mcW}{\mathcal W}
\newcommand{\mcX}{\mathcal X}
\newcommand{\mcWE}{{\mathcal W}{\mathcal E}}
\newcommand{\mcWKC}{{\mathcal W}{\mathcal K}{\mathcal C}}

\makeatletter
\@namedef{subjclassname@2020}{\textup{2020} Mathematics Subject Classification}
\makeatother

\begin{document}
\footskip30pt
\date{}

\title{$0$-dimensional ideal approximation theory}

\author{H.Y.\ ZHU}
\address{School of Mathematics, Nanjing University, Nanjing, CHINA}
\email{hongyuzhu@smail.nju.edu.cn}
\author{X.H.\ FU}
\address{School of Mathematics and Statistics, Northeast Normal University, Changchun, CHINA}
\email{fuxianhui@gmail.com}
\author{I.\ HERZOG}
\address{The Ohio State University at Lima, Lima, Ohio 45804 USA}
\email{iherzog@lima.ohio-state.edu}
\author{K.\ SCHLEGEL}
\address{University of Stuttgart, Institute of Algebra and Number Theory, Pfaffenwaldring 57, 70569 Stuttgart, GERMANY}
\email{kevin.schlegel@iaz.uni-stuttgart.de}



\thanks{The second and third authors were supported by NSFC No.12071064.}

\begin{abstract}
We propose axioms for a 0-dimensional version of ideal approximation theory. We note that extriangulated categories satisfy these axioms.

\end{abstract}

\maketitle

\section{Introduction}\label{S: Intro}

The familiar notion of an ideal of a ring generalizes nicely to any additive category $\mcA$ as a subbifunctor of $\Hom_{\mcA}(-,-)$ that distinguishes a class of morphisms. As in the classical case, one has the usual lattice operations of $I + J$ and $I \cap J$ on a given pair of ideals $I$ and $J$, but one can also use composition of morphisms to define the product $IJ$. This inevitably leads to notions like the left or right annihilator of an ideal, which, in the context of an additive category, emerges as a theory of $\Hom$-orthogonality of ideals, a generalization of the classical theory of torsion pairs.

In this paper, we propose an axiomatic treatment of additive categories that supports a theory of {\em ideal torsion pairs}. It emulates Quillen's introduction of an exact category, which provides the proper axiomatic context for the theory of $\Ext$-orthogonality for ideals, centered around {\em ideal cotorsion pairs}. Just as the axioms for an exact category distinguish a collection $\mcE$ of kernel-cokernel pairs in $\mcA$ called {\em conflations,} our axioms distinguish a collection $\mcWKC$ of weak kernel-cokernel pairs called {\em weak conflations} of an additive category.

The fundamental theorem of ideal cotorsion pairs in an exact category $(\mcA ; \mcE)$ is Salce's Lemma, which states that if $(\mcI, \mcJ)$ is an ideal cotorsion pair, i.e., a maximal pair of $\Ext$-orthogonal ideals, then $\mcI$ is $\Ext$-special precovering if and only if $\mcJ$ is $\Ext$-special preenveloping (we add the prefix $\Ext$ to distinguish from the notions of $\Hom$-special approximations which will be introduced in this paper). Recall that if $\mcI$ is an ideal (and $(\mcA ; \mcE)$ has enough projective objects), then an {\em $\Ext$-special $\mcI$-precover} of an object $X$ is a morphism $i \colon T \to X$ that arises as part of a commutative diagram
\begin{equation} \begin{split} \label{Eq:special precover}
\xymatrix@R=30pt@C=30pt{
       \Omega(X) \ar[r]\ar[d]_{\Omega(j)} & P \ar[r]\ar[d]  & X \ar@{=}[d]  \\
       \Omega(F) \ar[r]                   & T \ar[r]^{i}    & X}
\end{split} \end{equation}
of conflations, where $\Ext(\mcI, \Omega(j))=0$. The ideal $\mcI$ is said to be {\em $\Ext$-special precovering} if every object has an $\Ext$-special $\mcI$-precover. This ideal generalization of Salce's classical result \cite{SL} is the origin of {\em ideal approximation theory} \cite{FGHT} and because it is about $\Ext$ it may be regarded more specifically as $1$-dimensional ideal approximation theory, and we will therefore denote by $\mcI^{\perp_{1}}$ the right $\Ext$-orthogonal ideal of $\mcI$.
Ideal approximation theory was used in~\cite[Corollary 9.4]{FH} to settle the Benson-Gnacadja Conjecture~\cite{BG} for finite group rings $kG$ by providing a bound on the nilpotency index of the phantom ideal in the stable category $kG$-$\underline{\rm Mod}.$
 Recently, Breaz and Modoi \cite{BM} introduced ideal approximation theory to a triangulated category by replacing the functor $\Ext(-,-)$ by $\Hom(-,\Sigma(-))$, where $\Sigma$ is the shift endofunctor, and Asadollahi and Sadeghi \cite{AS0} introduced a higher version of this theory into an $n$-exact category.

The axioms we propose in this paper are meant to provide a context $(\mcA; \mcWKC)$, which we call a {\em weak kernel-cokernel category,} in which to develop $0$-dimensional ideal approximation theory. An $\mcI$-precover $i \colon T \to X$ is said to be $\Hom$-special, if it is part of a weak conflation
\begin{equation} \label{Eq:special 0-precover}
\xymatrix@R=30pt@C=30pt{T \ar[r]^i & X \ar[r]^j & F}
\end{equation} in $\mcWKC$. A $\Hom$-special $\mcI$-preenvelope is defined dually and this allows us to formulate a $0$-dimensional version of Salce's Lemma. \bigskip

\noindent {\bf Theorem~\ref{T: SL}.} {\em Let $(\mcI, \mcJ)$ be an ideal torsion pair in a weak kernel-cokernel category $(\mcA; \mcWKC)$. If
$$\xymatrix@R=30pt@C=30pt{T \ar[r]^{i} & X \ar[r]^{j} & F}$$
is a weak conflation, with $i \in \mcI$ and $j \in \mcJ$, then $i \colon T \to X$ is a $\Hom$-special $\mcI$-precover if and only if $j \colon X \to F$ is a $\Hom$-special $\mcJ$-preenvelope.}
\bigskip

\noindent Ideal torsion pairs $(\mcI, \mcJ)$ for which every object has $\Hom$-special approximations are called {\em complete;} they are $0$-dimensional analogues of complete ideal cotorsion pairs.

The advantage of working with ideals of morphisms rather than subcategories of objects is that we may exploit an idea first explicitly used by Christensen \cite{C}, the existence of the product $\mcI\mcJ$ of ideals. The second and third authors \cite{FH} proved that if $\mcI$ and $\mcJ$ are $\Ext$-special preenveloping in an exact category $(\mcA ; \mcE)$ with enough projective morphisms, then so is the product $\mcI\mcJ$. Instrumental to the proof is the determination of the left $\Ext$-orthogonal ideal of $\mcI\mcJ$, which consists of extensions $j' \star i'$ of morphisms $i' \in {^{\perp_{1}}\mcI}$ and $j' \in {^{\perp_{1}}\mcJ}$. Such an extension is defined to be a composition $j' \star i' = e^1 e^2$ that arises as part of a commutative diagram
\begin{equation} \begin{split} \label{Eq:mor ext}
\xymatrix@R=30pt@C=30pt{
                              & F\ar[d]_{e^{2}}\ar[dr]^{i'} \\
      X\ar[r]^{k}\ar[dr]_{j'} & Y\ar[d]^{e^{1}}\ar[r]^(.45){c}   & Z \\
                              & T,}
\end{split} \end{equation}
where the horizontal sequence is a conflation in $\mcE$. 
Section 4 is devoted to the definition and elementary properties of a {\em weak exact category.} This is a weak kernel-cokernel category $(\mcA; \mcWE)$ in which weak inflations (resp., weak deflations) are closed under composition and the Five Lemma is taken as an axiom. {\em Weak extensions} of morphisms are then defined as in diagram (\ref{Eq:mor ext}), but with the conflation in $\mcE$ replaced by a weak conflation in $\mcWE$. Weak exact categories satisfy the following $0$-dimensional version of Christensen's Lemma.
\bigskip

\noindent {\bf Theorem~\ref{T: FH in wec}.} {\em If $(\mcI_{1}, \mcJ_{1})$ and $(\mcI_{2}, \mcJ_{2})$ are complete ideal torsion pairs in a weak exact category $(\mcA; \mcWE)$, then so are $(\mcI_{1}\mcI_{2}, {\mcJ_{2}}\diamond{\mcJ_{1}})$ and $({\mcI_{1}}\diamond{\mcI_{2}}, \mcJ_{2}\mcJ_{1})$.
}
\bigskip

\noindent All extriangulated categories (Example~\ref{E: extri cat}) are examples of weak exact categories. A $0$-dimensional version of Wakamatsu's Lemma (Theorem~\ref{P: min ob ide}) is another feature of weak exact categories.
\bigskip

Approximation theory relative subcategories (rather than ideals) of module categories originates in the work of Auslander and Smal\o~\cite{AS} and Enochs and Jenda~\cite{EJ, GT}. These treatments of the theory reflect two different schools of thought distinguished by the size, skeletally small or not, of the ambient module category. In the work of Auslander and Smal\o, the theory is developed for subcategories of the category $\Rmod$ of finitely presented modules; in that of Enochs and Jenda, for the category
$\RMod$ of all modules. The difference brings to light a foundational issue (see Remark~\ref{R:fi}) that arises in the big category $\RMod$: ideals are already proper classes, so it is not possible without making use of some form of higher type theory to discuss lattices of ideals, which is no problem in the little category $\Rmod.$ In an effort to bridge these two methodologies, we introduce in Subsection 5.2 the finiteness condition of {\em image maximality} on an ideal, which has already been implemented in the work of the fourth author \cite{Sch}.

The last section is devoted to Frobenius categories $(\mcA; \mcE),$ where the respective $0$-dimensional approximation theory of the triangulated category $\underline{\mcA}$ is related to the $1$-dimensional approximation theory of $(\mcA; \mcE)$ by the shift endofunctor $\Sigma$ as follows. \bigskip

\noindent {\bf Theorem~\ref{T: bij com in F}.} {\em Let $(\mcA; \mcE)$ be a Frobenius category. An ideal cotorsion pair $(\mcI, \mcJ)$ in $(\mcA;\mcE)$ is complete if and only if the corresponding ideal torsion pair $(\pi(\mcI), \pi(\Sigma(\mcJ)))$ in $\underline{\mcA}$ is complete.} \bigskip

\noindent  The way the inverse endofunctor $\Omega$ relates $\Ext$-special precovers in $(\mcA; \mcE)$ with the $\Hom$-special precovers in $\underline{\mcA}$ is seen by a comparison of the diagrams (\ref{Eq:special precover}) and (\ref{Eq:special 0-precover}) above.

\section{Ideal Torsion Pairs}  \label{S:itp}
In this preliminary section, we recall the notion of an ideal of an additive category. This is a generalization of the classical concept of an ideal of a ring (see Example~\ref{E:ring ideals} below) so we review the operations on such ideals, as described in Stenstr\"{o}m~\cite{S}, in the setting of an additive category, which we assume is closed under finite coproducts.

Let $\mcA$ be an additive category and denote by $\Hom\colon \mcA^{\textrm{op}}\times\mcA\to \Ab$ the additive bifunctor, which associates to every pair $(A,B)$ of objects in $\mcA$ an abelian group $\Hom(A,B)$. In order to define $\Hom$-orthogonality of morphisms below, let us elaborate on how $\Hom$ acts on morphisms. If $f \colon X \to A$ and $g \colon B \to Y$ are morphisms in $\mcA$, then the morphism $\Hom(f,g) \colon \Hom(A,B) \to \Hom(X,Y)$
of abelian groups sends an $\mcA$-morphism $x\colon A\to B$ to $gxf\colon X\to Y$. It may be factored according to the following commutative square
$$\xymatrix@R=50pt@C=40pt{
    \Hom(A,B)\ar[rr]^{\Hom(f,\,B)}\ar[rrd]^{\Hom (f,\,g)}\ar[d]_{\Hom(A,\,g)}
                                         && \Hom(X,\,B)\ar[d]^{\Hom(X,\,g)} \\
    \Hom(A,Y)\ar[rr]_{\Hom(f,\,Y)}       && \Hom(X,Y). }
$$
An {\em ideal} $\mcI$ of $\mcA$ is an additive subbifunctor $\mcI \colon \mcA^{\textrm{op}}\times\mcA\to \Ab$ of $\Hom$, that is, it associates to the pair $(A,B)$ of objects in $\mcA$ a subgroup $\mcI(A,B)\subseteq\Hom(A,B)$ preserved by $\Hom(f,g)$,
$$\Hom(f,g) \colon \mcI(A,B)\to \mcI(X,Y).$$

Let $X$ be an object in $\mcA$. We say that $X$ is an {\em object} of an ideal $\mcI$ if the identity morphism $1_{X}$ belongs to $\mcI$, that is, if $\mcI(X, X)=\Hom(X, X)$. The objects of an ideal $\mcI$ form a full additive subcategory, denoted by $\Ob(\mcI) \subseteq \mcA$. Given an additive subcategory $\mcX\subseteq\mcA$, we may associate to it the ideal $\mcI(\mcX)$ of morphisms that factor through some object in $\mcX$. An ideal $\mcI$ is called an {\em object ideal} if $\mcI=\mcI(\Ob(\mcI))$, that is, if $\mcI$ is generated by identity morphisms $1_{X}$, as $X$ ranges over the objects of $\mcI$.

Given a collection $\mcM$ of morphisms in $\mcA$, the ideal $\langle\mcM\rangle$ {\em generated} by $\mcM$ is the smallest ideal containing $\mcM$. The ideal generated by a morphism $f\colon X\to Y$ in $\mcA$ is given by
$$\langle\, f \, \rangle(A,B):=\{\, \Sigma_{s}\, g_{s}fh_{s} \,\mid\, h_{s}\colon A\to X,~ g_{s}\colon Y\to B \,\}.$$

\begin{exa}\label{E:ring ideals}\rm
Let $R$ be a ring and $\mcA=\projR$ the category of finitely generated projective right $R$-modules. If $\mcI$ is an ideal of $\mcA$, then
$\mcI(R,R) \normal \End(R)=R$ is a two-sided ideal of $R$ with
$\mcI=\langle\, \mcI(R, R) \,\rangle$. If $\mcI(R, R)$ is finitely generated as a two-sided ideal with generators $r_{1}, r_{2}, \ldots, r_{n}$, then
$\mcI=\langle\, \oplus_{s=1}^{n}\, r_{s}\colon R^{n}\to R^{n} \,\rangle$.
\end{exa}

\subsection{$\Hom$-orthogonality of morphisms}
A {\em torsion pair} $(\mcT, \mcF)$ in $\mcA$ consists of two subcategories $\mathcal{T}, \mathcal{F}\subseteq\mcA$, which is a maximal $\Hom$-orthogonal pair, that is,
$$\mathcal{T}=\{\, T\in\mcA \,\mid\, \Hom(T, \mathcal{F})=0 \,\} ~\;~\textrm{and}~\;~ \mathcal{F}=\{\, F\in\mcA \,\mid\, \Hom(\mathcal{T}, F)=0 \,\}.$$
A pair $(i,j)$ of morphisms is called {\em $\Hom$-orthogonal} if $\Hom(i,j)=0$, and a pair $(\mcI, \mcJ)$ of ideals is {\em $\Hom$-orthogonal} if every pair $(i,j)$ of morphisms, with $i\in\mcI$ and $j\in\mcJ$, is $\Hom$-orthogonal. Given a collection $\mcM$ of morphisms in $\mcA$, the {\em right $\Hom$-orthogonal ideal} of $\mcM$ is defined to be
$$\mcM^{\perp}:= \{\,j \,\mid\, \Hom(f,j)=0 ~\textrm{for~all}~ f\in\mcM \,\}.$$
The {\em left $\Hom$-orthogonal ideal} $^{\perp}\mcM$ is defined dually\footnote{A general treatment would use notation like $\mcM^{\perp_{0}}$ and $^{\perp_{0}}\mcM$ to distinguish $\Hom$-orthogonality from its higher analogues.}. It is clear that a pair $(\mcI, \mcJ)$ of ideals is $\Hom$-orthogonal if and only if $\mcI\subseteq {^{\perp}\mcJ}$ (if and only if $\mcJ\subseteq \mcI^{\perp}$).

\begin{df}\label{D: itp}
A pair $(\mcI, \mcJ)$ of ideals in $\mcA$ is an {\em ideal torsion pair} if $\mcI={^{\perp}\mcJ}$ and $\mcJ=\mcI^{\perp}$.
If $\mcM$ is a collection of morphisms, then the ideal torsion pair {\em generated} by $\mcM$ is given by $(^{\perp}({\mcM^{\perp}}), \mcM^{\perp})$; the ideal torsion pair {\em cogenerated} by $\mcM$ is defined dually.
\end{df}

For example, if $(\mcT, \mcF)$ is a torsion pair in $\mcA$, then $(\mcI(\mcT), \mcI(\mcF))$ is  an ideal torsion pair.

\subsection{Ideal annihilators}
The {\em left annihilator} of an ideal $\mcI$ of $\mcA$ is the ideal
$$\ell(\mcI):=\{\,j \,\mid\, ji=0~ \textrm{for~all}~ i\in\mcI~
\textrm{with}~(i,j)~\textrm{composable} \,\}.$$
The rule $\ell\colon \mcI\mapsto \ell(\mcI)$ is inclusion-reversing, as is the rule $r\colon \mcI\mapsto r(\mcI)$ given by the dually defined {\em right annihilator}.

\begin{thm}\label{T: orth and ann}
Let $\mcI$ be an ideal of $\mcA$. Then $\mcI^{\perp}=\ell(\mcI)$ and $^{\perp}\mcI=r(\mcI)$. A pair $(\mcI, \mcJ)$ of ideals in $\mcA$ is therefore an ideal torsion pair if and only if $\mcI=r(\mcJ)$ and $\mcJ=\ell(\mcI)$.
\end{thm}

\begin{proof}
To prove the first equality, we note that $\mcI^{\perp}\subseteq\ell(\mcI)$. Conversely, let $j\colon B\to Y$ be a morphism in $\ell(\mcI)$, and let $i\colon X\to A$ be a morphism in $\mcI$. Then the morphism
$$\Hom(i,j)\colon \Hom(A,B)\to\Hom(X,Y)$$
sends every $x\in\Hom(A, B)$ to $jxi$. Note that $xi\in\mcI$ since $\mcI$ is an ideal, then $j(xi)=0$. It follows that $\Hom(i, j)=0$. The second equality is just the dual.
\end{proof}

We can use Theorem~\ref{T: orth and ann} to paraphrase Definition~\ref{D: itp} and that the ideal torsion pair generated by an ideal $\mcI$ is given by $(r(\ell(\mcI)), \ell(\mcI))$.

\subsection{Operations on ideals}
Let $\mcI_{1}$ and $\mcI_{2}$ be two ideals of $\mcA$. The {\em sum} of $\mcI_{1}$ and $\mcI_{2}$ is defined for a pair $(A,B)$ of objects in $\mcA$ to be
$$(\mcI_{1}+\mcI_{2})(A,B) := \mcI_{1}(A,B)+\mcI_{2}(A,B),$$
which is the smallest ideal containing $\mcI_{1}$ and $\mcI_{2}$; the {\em intersection} $\mcI_{1}\cap\mcI_{2}$ is defined similarly, which is the largest ideal contained in $\mcI_{1}$ and $\mcI_{2}$. The following states the familiar fact that left (resp., right) annihilator ideals are closed under intersection.

\begin{coro}\label{C: orth of sum}
If $\mcI_{1}$ and $\mcI_{2}$ are ideals of $\mcA$, then $\ell(\mcI_{1}+\mcI_{2})=\ell(\mcI_{1}) \cap \ell(\mcI_{2})$, and dually, $r(\mcI_{1}+\mcI_{2})=r(\mcI_{1}) \cap r(\mcI_{2})$.
\end{coro}

\begin{proof}
Let us check that $\ell(\mcI_{1}+\mcI_{2})=
\ell(\mcI_{1})\cap\ell(\mcI_{2})$. It is clear that  $\ell(\mcI_{1})\cap\ell(\mcI_{2})\subseteq
\ell(\mcI_{1}+\mcI_{2})$. For the converse, we note that $\mcI_{1}\subseteq\mcI_{1}+\mcI_{2}$, and so $\ell(\mcI_{1}+\mcI_{2})\subseteq \ell(\mcI_{1})$. It follows that $\ell(\mcI_{1}+\mcI_{2})\subseteq \ell(\mcI_{1})\cap \ell(\mcI_{2})$. The second equality is proved dually.
\end{proof}

Given two ideals $\mcI_{1}$ and $\mcI_{2}$ of $\mcA$, let us observe that the {\em product}
$$\mcI_{1}\mcI_{2}:= \{\, i_{1}i_{2} \,\mid\, i_{1}\in \mcI_{1},~ i_{2}\in\mcI_{2}, ~\textrm{and}~ (i_{2},i_{1})~\textrm{composable} \,\}$$
is also an ideal. It suffices to verify that the sum of parallel composable pairs $X\stackrel{i_{2}}{\longrightarrow}Y\stackrel{i_{1}}{\longrightarrow}Z$, with $i_{2}\in\mcI_{2}$ and $i_{1}\in\mcI_{1}$, and  $X\stackrel{i'_{2}}{\longrightarrow}Y'\stackrel{i'_{1}}{\longrightarrow}Z$, with
$i'_{2}\in\mcI_{2}$ and $i'_{1}\in\mcI_{1}$, also belongs to the product. But the sum $i_{1}i_{2}+i'_{1}i'_{2} \colon X \to Z$ can be viewed as a composition of
$$\xymatrix@C=55pt{ X \ar[r]^-{\left(\begin{smallmatrix}i_{2}\\i'_{2}\end{smallmatrix}\right)} & Y \oplus Y' \ar[r]^-{(i_{1},i'_{1})} & Z,
}$$
where $\left( \begin{smallmatrix}  i_2  \\  i'_2  \end{smallmatrix}  \right) =\left(  \begin{smallmatrix} 1 \\ 0  \end{smallmatrix}\right) i_2 +\left(\begin{smallmatrix}0\\1\end{smallmatrix}\right)i'_2$ belongs to $\mcI_{2}$ and
$(i_{1}, i'_{1})=i_{1}(1,0)+ i'_{1}(0,1)$ belongs to $\mcI_{1}$.

\begin{exa}\label{E: Z(p2)}\rm
If $\mcI$ is an ideal of $\mcA$, then $\ell(\mcI)\,\mcI=0$ and similarly $\mcI\,(r(\mcI))=0$. Theorem \ref{T: orth and ann} implies that one always has that $(\mcI\cap\mcJ)^{2}=0$ for an ideal torsion pair $(\mcI, \mcJ)$, but it may happen, in contrast to the classical theory of torsion pairs, that $\mcI\cap\mcJ \neq 0$. Indeed, if $p$ is a prime number, then $(\langle p \rangle, \langle p \rangle)$ is an ideal torsion pair in ${\rm proj}\mbox{-}\mathbb{Z}(p^2)$.
\end{exa}

The {\em left conductor} of an ideal $\mcJ$ into an ideal $\mcI$ is given by
$$\ell(\mcI\colon \! \mcJ) := \{\, f \,\mid\, fj\in\mcI~\textrm{for~all}~
j\in\mcJ~\textrm{with}~(j,f)~\textrm{composable} \,\}.$$
Thus $\ell(\mcI\colon \! \mcJ)$ is the maximum ideal satisfying the inclusion
$\ell(\mcI\colon \! \mcJ)\,\mcJ\subseteq\mcI$. The {\em right conductor} is defined similarly, so that $\mcI\,(r(\mcI\colon \! \mcJ)) \subseteq\mcJ$. We note that $\ell(\mcI_{1}\mcI_{2})\,\mcI_{1}\subseteq\ell(\mcI_{2})$ for two ideals $\mcI_{1}$ and $\mcI_{2}$ of $\mcA$, so that the following equalities hold.

\begin{coro}\label{C: orth of prod}
If $\mcI_{1}$ and $\mcI_{2}$ are ideals of $\mcA$, then
$\ell(\mcI_{1}\mcI_{2})=\ell(\ell(\mcI_{2})\colon\! \mcI_{1})$ while $r(\mcI_{1}\mcI_{2})=r(\mcI_{2}\colon \! r(\mcI_{1}))$.
\end{coro}

\begin{rem} \rm \label{R:fi}
A foundational issue arises when we attempt to generalize the ideal lattice of a ring to the setting of an additive category. While the ring may be regarded as a preadditive category with one object, the additive category $\mcA$ need not even be skeletally small. In that case, an ideal $\mcI \colon \mcA^{\textrm{op}} \times \mcA \to \Ab$ is a {\em class} function. It therefore becomes problematic, without the use of {\em higher-order types,} to index a collection of ideals or to consider the lattice of all ideals of an additive category.
\end{rem}

\section{Complete Ideal Torsion Pairs}

In this section, we develop the theory of complete ideal torsion pairs in the setting of an additive category $\mcA$ equipped with a weak kernel-cokernel structure.

\subsection{Weak kernels and cokernels} Given a morphism $g\colon Y\to Z$ in $\mcA$, a {\em weak kernel} of $g$ is a morphism $k\colon X\to Y$ satisfying (1) $gk=0$; and (2) if $k'\colon X'\to Y$ is a morphism for which $gk'=0$, then $k'$ factors as indicated by the commutative diagram
$$\xymatrix@R=30pt@C=30pt{
                    & X' \ar[d]_{k'}\ar[rd]^{0}\ar@{-->}[ld] \\
      X \ar[r]^{k}  & Y \ar[r]^(.45){g}    & Z.}
$$
A {\em kernel} of $g$ is a weak kernel for which the factorization is unique. {\em Weak cokernels} and {\em cokernels} are defined dually.

\begin{exa}\label{E: weak kernels}\rm
The category $\projR$ has weak kernels if and only if the ring $R$ is right coherent. This is equivalent, via the duality $\Hom_{R}(-, R)\colon (\projR)^{\textrm{op}}\to \Rproj$, to the condition that the category $\Rproj$ has weak cokernels.
\end{exa}

\begin{exa} \label{E: weak ker ab} \rm
In an abelian category, every morphism $f \colon Y \to Z$ has a kernel $k \colon K \to Y$, so if $k' \colon K' \to Y$ is a weak kernel, then the uniqueness property of the kernel implies that the composition of vertical arrows on the left in the commutative diagram
$$\xymatrix@R=30pt@C=30pt{
 K \ar[r]^k \ar[d]^s     & Y \ar[r]^f \ar@{=}[d] & Z \ar@{=}[d]  \\
 K' \ar[r]^{k'} \ar[d]^p & Y \ar[r]^f \ar@{=}[d] & Z \ar@{=}[d] \\
 K \ar[r]^k              & Y \ar[r]^f            & Z}
$$
is given by $1_K = ps$. Conversely, any morphism $g \colon L \to Y$ such that $\im g = \Ker f = K$ and for which there exists a section $s \colon K \to L$, $gs = k$, is a weak kernel of $f$.
\end{exa}

The characterization of $\ell(\mcI_{1}\mcI_{2})$ given in Corollary \ref{C: orth of prod} serves to verify the following.

\begin{thm}\label{T: cond element}
Let $\mcI_{1}$ and $\mcI_{2}$ be two ideals of $\mcA$ and suppose there exists a commutative diagram
$$\xymatrix@R=30pt@C=30pt{
                                   & F_{1} \ar[d]_{a^{2}}\ar[dr]^{j_{1}}\\
       X \ar[r]^k  \ar[dr]_{j_{2}} & Y \ar[d]^{a^{1}}\ar[r]^(.45){g}        & Z \\
                                   & F_{2},}
$$
where $k$ is a weak kernel of $g$, $j_{1}\in\ell(\mcI_{1})$, and $j_{2}\in\ell(\mcI_{2})$. Then $a=a^{1}a^{2}\in\ell(\mcI_{1}\mcI_{2})$.
\end{thm}

\begin{proof}
Since $\ell(\mcI_{1}\mcI_{2})=\ell(\ell(\mcI_{2})\colon\! \mcI_{1})$, it suffices to show that if $h\colon H\to F_{1}$ belongs to $\mcI_{1}$, then $ah=a^{1}a^{2}h\in\ell(\mcI_{2})$. But $ga^{2}h=j_{1}h=0$ so that $a^{2}h$ factors through the weak kernel of $g$,
$$\xymatrix@R=30pt@C=40pt{
                                 & H \ar[d]_(.55){h} \ar[rdd]^0 \ar@{-->}[ldd] \\
                                 & F_{1}\ar[d]_{a^{2}} \ar[dr]^(.4){j_{1}}\\
       X \ar[r]^k  \ar[dr]_{j_2} & Y \ar[d]^{a^{1}} \ar[r]^(.45){g}   &    Z \\
                                 & F_{2},}
$$
and $ah$ factors through $j_{2}$, as required.
\end{proof}

A ({\em weak}) {\em kernel-cokernel pair} in $\mcA$ is a composable pair $(k,c)$ of morphisms
$$\xymatrix@R=30pt@C=30pt{
    X\ar[r]^{k}  & Y\ar[r]^{c}  & Z, }
$$
such that $c$ is a (weak) cokernel of $k$ and $k$ is a (weak) kernel of $c$. The weak kernel-cokernel pairs of an additive category $\mcA$ themselves form an additive category $\mcWKC (\mcA)$ whose morphisms are given by triples of arrows $(x,y,z)$ for which  the diagram
$$\xymatrix@R=30pt@C=30pt{
    X \ar[r]^{k} \ar[d]^x & Y \ar[r]^{c} \ar[d]^y  & Z \ar[d]^z \\
    X' \ar[r]^{k'}        & Y' \ar[r]^{c'}         & Z'
}$$
commutes.

\subsection{Weak kernel-cokernel structures}
Let us now consider an additive category $\mcA$ together with a distinguished collection $\mcWKC \subseteq \mcWKC (\mcA)$ of weak kernel-cokernel pairs
$$\xymatrix@R=30pt@C=30pt{
    X\ar[r]^{k}  & Y\ar[r]^{c}  & Z, }
$$
called {\em weak conflations}. A morphism $k$ is called a {\em weak inflation} if it admits a weak kernel-cokernel pair $(k,c) \in \mcWKC$, and a {\em weak deflation} is defined dually.

\begin{df}\label{D: wkc}
A {\em weak kernel-cokernel structure} $\mcWKC$ on an additive category $\mcA$ is a distinguished collection of weak kernel-cokernel pairs satisfying:  \vspace{.3em}

$(\WE_0)$ $\mcWKC \subseteq \mcWKC (\mcA)$ is an additive subcategory closed under isomorphism; \vspace{.3em}

$(\WE_1)$ for all objects $A\in\mcA$, the kernel-cokernel pair $A\stackrel{1_{A}}{\longrightarrow}A\stackrel{0}{\longrightarrow}0$ is a weak conflation; and \vspace{.3em}

$(\WE_1)^{\rm{op}}$ for all objects $A\in\mcA$, the kernel-cokernel pair $0\stackrel{0}{\longrightarrow}A\stackrel{1_{A}}{\longrightarrow}A$ is a weak conflation. \vspace{.3em}

An additive category $\mcA$ equipped with a weak kernel-cokernel structure $\mcWKC$ is called a {\em weak kernel-cokernel category} denoted by $(\mcA; \mcWKC)$.
\end{df}

The statement in Axiom $(\WE_0)$ that $\mcWKC$ is an additive category implies that weak conflations are closed under finite coproducts. \textbf{We will assume throughout the remainder of this section that $(\mcA; \mcWKC)$ is a weak kernel-cokernel category}.

\begin{exa} \label{E: max WKC} \rm
For every additive category $\mcA$, the collection $\mcWKC (\mcA)$ in which every weak kernel-cokernel pair is distinguished is a weak kernel-cokernel structure on $\mcA$.
\end{exa}

An additive functor $F \colon (\mcA; \mcWKC) \to (\mcA'; \mcWKC')$ of weak kernel-cokernel categories will be called {\em weak exact} if it respects weak conflations.

\begin{exa} \label{E: trivial} \rm
Every additive category $\mcA$ admits a unique minimal weak kernel-cokernel structure $ \mcWKC_0$ whose weak conflations are the split kernel-cokernel pairs. These weak conflations actually constitute an exact structure on $\mcA$. The weak kernel-cokernel category $(\mcA; \mcWKC_0)$ is minimal in the sense that if $(\mcA'; \mcWKC')$ is a weak kernel-cokernel category and $F \colon \mcA \to \mcA'$ an additive functor, then there is a unique weak exact functor $\tilde{F}$ that solves the commutative diagram
$$\xymatrix@R=35pt@C=35pt{(\mcA; \mcWKC_0) \ar@{-->}[rd]^-{\tilde{F}} \\
\mcA \ar@{^{(}->}[u] \ar[r]^-{F} & (\mcA'; \mcWKC').}
$$
\end{exa}

\subsection{Weak extensions of ideals}
Theorem~\ref{T: cond element} serves as the inspiration for the notion of a {\em weak extension} of morphisms, a weak variation of the concept of morphism extension \cite[Section 4]{FH} introduced by the second and third authors for an exact category.

\begin{df}\label{D: ext}
A morphism $e\colon F\to T$ in $(\mcA; \mcWKC)$ is said to be a {\em weak extension} of $j$ by $i$, denoted by
$e=i\star j$, if there is a commutative diagram
$$\xymatrix@R=30pt@C=30pt{
                             & F\ar[d]_{e^{2}}\ar[dr]^{j}\\
      X\ar[r]^{k}\ar[dr]_{i} & Y\ar[d]^{e^{1}}\ar[r]^(.45){c}   & Z \\
                             & T,}
$$
where $e=e^{1}e^{2}$ and the horizontal sequence is a weak conflation.
\end{df}

For example, the identity morphism $1_{X}$ of the object $X$ in the middle of a weak conflation $\xymatrix@1{T \ar[r]^{i} & X \ar[r]^{j} & F}$ may be regarded as a weak extension of $j$ by $i$, $1_X = i \star j$, as in
$$\xymatrix@R=30pt@C=30pt{
                                 & X \ar@{=}[d] \ar[dr]^{j} \\
       T  \ar[r]^{i} \ar[dr]_{i} & X \ar@{=}[d] \ar[r]^(.45){j}   & F \\
                                 & X.
}$$

\begin{prop}
If $\mcI$ and $\mcJ$ are ideals of $(\mcA; \mcWKC)$, then the collection $\mcI \diamond \mcJ$ of weak extensions $i \star j$ of morphisms $j \in \mcJ$ by $i \in \mcI$ forms an ideal of $(\mcA; \mcWKC)$ that contains $\mcI + \mcJ$. It is called the {\em weak extension ideal} of $\mcJ$ by $\mcI$.
\end{prop}
\begin{proof}
First note that weak extensions $i \star j$ are closed under left and right multiplication, $f(i \star j)g = (fi) \star (jg)$. It suffices therefore to verify that the sum of parallel weak extensions $i \star j$, $i' \star j' \colon F \to T$ is itself such a weak extension. This follows using the same argument that shows a product of ideals is closed under sum (see~\cite[Lemma 4.1]{FH}) together with $(\WE_0)$ that weak conflations are closed under finite coproducts.

Lastly, observe that if $i \colon X \to T$ is a morphism in $\mcI$, then $i = i \star 0 \in \mcI \diamond  \mcJ$, as indicated by the commutative diagram
$$\xymatrix@R=30pt@C=30pt{
                                 & X \ar@{=}[d] \ar[dr]^0 \\
     X  \ar[r]^{1_X} \ar[dr]_{i} & X \ar[d]^i \ar[r]^0    & 0 \\
                                 & T}
$$
with middle row a weak conflation by $(\WE_1)$. Thus $\mcI \subseteq \mcI \diamond \mcJ$. That $\mcJ \subseteq \mcI \diamond \mcJ$ follows by a dual argument and therefore $\mcI + \mcJ \subseteq \mcI \diamond \mcJ$.
\end{proof}

This proposition tells us that the weak extension is an operation on ideals of $(\mcA; \mcWKC)$ that can be used to rephrase Theorem \ref{T: cond element} as follows, which is a $0$-dimensional analogue of \cite[Theorem 5.3]{FH}.

\begin{coro}\label{C: prod perp inclusion}
If $\mcI_{1}$ and $\mcI_{2}$ are ideals of $(\mcA; \mcWKC)$, then $\ell(\mcI_{2})\diamond\ell(\mcI_{1})\subseteq \ell(\mcI_{1}\mcI_2)$ and
$r(\mcI_{2})\diamond r(\mcI_{1})\subseteq r(\mcI_{1}\mcI_2)$.
\end{coro}

\subsection{Ideal approximations}
Let $\mathcal{I}$ be an ideal of $\mathcal{A}$. An $\mathcal{I}$-{\em precover} of an object $X\in\mcA$ is a morphism $i\colon T\to X$ in $\mathcal{I}$ such that any other morphism $i'\colon T'\to X$ in $\mathcal{I}$ factors through $i$ as indicated by the dotted arrow
$$\xymatrix@R=30pt@C=30pt{
                  & T'\ar@{-->}[ld]_{}\ar[d]^{i'}\\
      T\ar[r]^{i} & X. }
$$
The ideal $\mathcal{I}$ is called {\em precovering} in $\mathcal{A}$ if every object $X\in\mathcal{A}$ has an $\mathcal{I}$-precover. The notion of a {\em preenveloping ideal} is defined in a dual manner and we say that an ideal is {\em approximating} if it is precovering or preenveloping.

\begin{exa}\rm
In the category $\projR$, an ideal $\mcI$ is precovering if and only if $\mcI(R,R) \normal R$ is finitely generated as a right ideal, $\mcI(R,R)=\Sigma_{s=1}^{n} \, r_{s}R$, in which case the homomorphism
$$(r_{1}, \ldots ,r_{n})\colon R^{n} \to R, \; \left(\begin{array}{c} x_1 \\ x_2 \\ \vdots \\ x_n \end{array}\right) \mapsto r_1x_1 + r_2 x_2 + \cdots + r_n x_n,$$
is an $\mcI$-precover of the object $R$.
\end{exa}

The following result shows that the sum and the product are operations on approximating ideals.

\begin{prop} \label{P: sum and prod} {\rm (\cite[Proposition 1]{FHHZ} and \cite[Theorem 8.1]{FGHT})}
If $\mcI_{1}$ and $\mcI_{2}$ are precovering ideals in $\mcA$, then both $\mcI_{1}+\mcI_{2}$ and $\mcI_{1}\mcI_{2}$ are also precovering ideals. Precisely, given an object $X \in \mcA$, an $\mcI_1$-precover $i_1 \colon T_1 \to X$ and $\mcI_2$-precovers $i_2 \colon T_2 \to X$ and $i'_2 \colon T'_2 \to T_1$, then $(i_{1}, i_{2}) \colon T_{1} \oplus T_{2} \to X$ is an $\mcI_1 + \mcI_2$-precover and $i_1 i'_2 \colon T'_2 \to X$ is an $\mcI_1 \mcI_2$-precover. Dually, if $\mcJ_{1}$ and $\mcJ_{2}$ are preenveloping ideals in $\mcA$, then so are $\mcJ_{1}+\mcJ_{2}$ and $\mcJ_{1}\mcJ_{2}$.
\end{prop}

The following is a $0$-dimensional version of the Ideal Salce Lemma.

\begin{thm} \label{T: SL}
Let $(\mcI, \mcJ)$ be an ideal torsion pair in $(\mcA; \mcWKC)$ and $X \in \mcA$ an object. The following are equivalent for a weak conflation
$$\xymatrix@R=30pt@C=30pt{T \ar[r]^{i} & X \ar[r]^{j} & F:}$$
\begin{enumerate}
\item $i\colon T\to X$ is an $\mcI$-precover of $X;$
\item $i \in \mcI$ and $j \in \mcJ;$
\item $j\colon X\to F$ is a $\mcJ$-preenvelope of $X;$ and
\item $i\colon T\to X$ is an $\mcI$-precover of $X$ and $j\colon X\to F$ is a $\mcJ$-preenvelope of $X$.
\end{enumerate}
\end{thm}
\begin{proof}
We only prove ($1$) $\Leftrightarrow$ ($2$), the proof will refer to the following commutative diagram
$$\xymatrix@R=30pt@C=30pt{
               & T' \ar[rd]^{0} \ar@{-->}[ld] \ar[d]^{i'} \\
  T \ar[r]^{i} & X \ar[r]^{j}  & F.}
$$
To prove the forward implication, let $i' \in \mcI$ be a morphism for which the pair $(i',j)$ is composable. But then $i'$ factors through $i$, and $ji' = 0$. Thus $j \in \ell (\mcI) = \mcJ$. For the converse, suppose that $i' \colon T' \to X$ is in $\mcI = r(\mcJ)$. Then $ji' = 0$, so the weak kernel property of $i$ ensures that $i'$ factors through it.
\end{proof}

\begin{df} \label{D: Hom special}
Let $\mcI$ be an ideal of $(\mcA; \mcWKC)$ and $X \in \mcA$ an object. An $\mcI$-precover $i \colon T \to X$ of $X$ is called $\Hom$-{\em special} if it is a weak inflation. Dually, an $\mcI$-preenvelope of $X$ is $\Hom$-{\em special} if it is a weak deflation. Similarly, an ideal $\mcI$ of $(\mcA; \mcWKC)$ is $\Hom$-{\em special precovering} if every object $X \in \mcA$ has a $\Hom$-special $\mcI$-precover; dually, it is $\Hom$-{\em special preenveloping} if every object has a $\Hom$-special $\mcI$-preenvelope.
\end{df}

Depending on the reader's preference, the following consequence of Theorem~\ref{T: SL} may also be regarded as a $0$-dimensional version of the Ideal Salce Lemma.

\begin{coro} \label{C: Hom special}
If $(\mcI, \mcJ)$ is an ideal torsion pair in $(\mcA; \mcWKC)$, then $\mcI$ is $\Hom$-special precovering if and only if $\mcJ$ is $\Hom$-special preenveloping. Such ideal torsion pairs are called {\em complete.}
\end{coro}

The proof of the 1-dimensional ideal version of Salce's Lemma is a shifted version of Theorem~\ref{T: SL}. One uses the existence of an $\Ext$-special $\mcJ$-preenvelope of $\Omega (X)$ (see diagram (\ref{Eq:special precover})) to produce an $\Ext$-special $\mcI$-precover of $X$. The proof of Theorem~\ref{T: SL} only relies on the particular object $X \in \mcA$.

\begin{coro} \label{C: complete bijection}
The rule $\mcI \mapsto \ell(\mcI)$ is a bijective inclusion-reversing correspondence between $\Hom$-special precovering ideals and $\Hom$-special preenveloping ideals of $(\mcA; \mcWKC)$; the inverse rule is given by $\mcJ \mapsto r(\mcJ)$.
\end{coro}

\begin{exa}\rm
Let $\mcA=\mathscr{T}$ be a triangulated category with the shift functor $\Sigma$. Note that the collection of distinguished triangles forms a weak kernel-cokernel structure on $\mathscr{T}$. Recall that a {\em projective class} \cite{C} is a pair $(\mathscr{P}, \mathscr{I})$ with $\mathscr{P}$ a subcategory closed under summands and $\mathscr{I}$ an ideal such that (1), $\Hom(P,i)=0$ for any object $P$ in $\mathscr{P}$ and any morphism $i$ in $\mathscr{I}$; and (2), for any object $X$ in $\mathscr{T}$, there is a triangle
$$\xymatrix@R=30pt@C=30pt{
				P\ar[r] & X\ar[r]^{i} & F\ar[r]  & \Sigma(P)}
$$
with $P\in\mathscr{P}$ and $i\in\mathscr{I}$. It is fairly easy to check that if $(\mathscr{P}, \mathscr{I})$ is a projective class in $\mathscr{T}$, then $(\mcI(\mathscr{P}), \mathscr{I})$ is a complete ideal torsion pair. Conversely, if $(\mcI,\mcJ)$ is a complete ideal torsion pair with $\mcI$ an object ideal, then $(\Ob(\mcI), \mcJ)$ is a projective class.
\end{exa}

The following proposition provides a sufficient condition for a precovering ideal $\mcI$ to generate a complete ideal torsion pair.

\begin{prop} \label{P: weak comp}
Let $\mcI$\ be a precovering ideal in $(\mcA; \mcWKC)$ such that every object $X \in \mcA$ has an $\mcI$-precover $i \colon T \to X$ which has a weak cokernel $j \colon X \to F$ that is a weak deflation. Then $j$ is a $\Hom$-special $\ell (\mcI)$-preenvelope of $X$. Consequently, $(r(\ell (\mcI)), \ell(\mcI))$ is a complete ideal torsion pair.
\end{prop}

\begin{proof}
To see that $j \colon X \to F$ belongs to $\ell (\mcI)$ apply the argument used to prove ($1$) $\Rightarrow$ ($2$) of Theorem \ref{T: SL}. Now suppose that $j' \colon X \to F'$ belongs to $\ell (\mcI)$. Then $j'i =0$ and we get the factorization
$$\xymatrix@R=30pt@C=30pt{
  T \ar[r]^i \ar[rd]_{0} & X \ar[r]^j \ar[d]^{j'} & F \ar@{-->}[dl] \\
                         & F'}
$$
from the weak cokernel property of $j$. As $j$ is a weak deflation, it is a $\Hom$-special $\ell (\mcI)$-preenvelope of $X$, as required.
\end{proof}

Every $\Hom$-special precovering ideal $\mcI$ of $(\mcA; \mcWKC)$ is part of a complete ideal torsion pair $(\mcI, \ell (\mcI))$, so it is evidently a right annihilator ideal. The following result is a $0$-dimensional version of \cite[Corollary 8.3]{FH}, which strengthens Corollary \ref{C: prod perp inclusion} to show that a weak extension of such ideals is itself a right annihilator ideal.

\begin{thm} \label{T: prod perp equ}
If $(\mcI_{1}, \mcJ_{1})$ and $(\mcI_{2}, \mcJ_{2})$ are complete ideal torsion pairs in $(\mcA; \mcWKC)$, then $\mcI_{1}\diamond\mcI_{2} = r(\mcJ_{2}\mcJ_{1})$ and dually
$\mcJ_{2}\diamond\mcJ_{1} = \ell(\mcI_{1}\mcI_{2})$.
\end{thm}
\begin{proof}
We only prove the second equality. Corollary \ref{C: prod perp inclusion} implies the inclusion $\mcJ_{2}\diamond\mcJ_{1} \subseteq \ell(\mcI_{1} \mcI_{2})$. To see the other inclusion, suppose that
$g\colon X\to F$ belongs to $\ell(\mcI_{1}\mcI_{2})$ and let us represent $g=j'_{2}\star j_{1}$ as a weak extension of morphisms, with $j_{1}\in\mcJ_{1}$ and
$j'_{2}\in\mcJ_{2}$. By Theorem \ref{T: SL}, the object $X$ admits a weak conflation
$$\xymatrix@C=30pt{T_{1} \ar[r]^{i_{1}} & X \ar[r]^{j_{1}} & F_{1}}$$
with $i_{1}\in\mcI_{1}$ and $j_{1}\in\mcJ_{1}$, and there is a weak kernel-cokernel pair, in fact, a weak conflation,
$$\xymatrix@C=30pt{T_{2} \ar[r]^{i_{2}} & T_{1} \ar[r]^{j_{2}} & F_{2}}$$
with $i_{2}\in\mcI_{2}$ and $j_{2}\in\mcJ_{2}$.
This gives rise to a commutative diagram
$$\xymatrix@R=30pt@C=30pt{
     T_{2}\ar[r]^{i_{2}}\ar@{=}[d] & T_{1}\ar[d]^{i_{1}}\ar[r]^{j_{2}} & F_{2}\ar@{-->}[d]^{f}\\
     T_{2}\ar[r]^{i_{1}i_{2}}      & X \ar[d]^{j_{1}}  \ar[r]^{g}      & F \\
                                   & F_{1}, }
$$
where the morphism $f\colon F_{2}\to F$ is induced by the weak cokernel property of $j_{2}$ and the hypothesis $g(i_{1}i_{2})=0$. But then there is the commutative diagram
$$\xymatrix@R=30pt@C=30pt{
                                          & X \ar@{=}[d]\ar[dr]^{j_{1}}    \\
     T_{1} \ar[r]^{i_{1}}\ar[dr]_{fj_{2}} & X \ar[d]^{g}\ar[r]^{j_{1}}  & F_{1}\\
                                          & F,}
$$
where the middle row is a weak conflation, $j_{1}\in\mcJ_{1}$, and
$fj_{2}\in\mcJ_{2}$. Thus we get that $g=fj_{2}\star j_{1} \in\mcJ_{2}\diamond \mcJ_{1}$.
\end{proof}

\subsection{Weak kernel-cokernel categories with enough weak inflations and weak deflations}
\begin{df} \label{D: enough}
A weak kernel-cokernel category $(\mcA; \mcWKC)$ is said to have {\em enough weak inflations} if every morphism $f \colon X \to Y$ in $\mcA$ has a weak kernel $k \colon K \to X$ which is a weak inflation. Dually, we say that $(\mcA; \mcWKC)$ has {\em enough weak deflations} if for every morphism $f$ as above, there exists a weak deflation which is a weak cokernel of $f$.
\end{df}

The following is an immediate consequence of Proposition \ref{P: weak comp}.

\begin{prop} \label{P: precover complete}
If $\mcI$ is a precovering ideal in a weak kernel-cokernel category $(\mcA; \mcWKC)$ with enough weak deflations, then the ideal torsion pair $(r(\ell (\mcI)), \ell(\mcI))$ is complete.
\end{prop}

If $\mathcal{I}_{1}$ and  $\mathcal{I}_{2}$ are precovering ideals, then Proposition \ref{P: sum and prod} implies that $\mathcal{I}_{1}+\mathcal{I}_{2}$ is a precovering ideal, and its left annihilator is given by $\ell(\mathcal{I}_{1}+\mathcal{I}_{2})=\ell(\mathcal{I}_{1})\cap\ell(\mathcal{I}_{2})$. Let us apply this observation to complete ideal torsion pairs.

\begin{coro}\label{C: int and sum in wac}
Let $(\mcA; \mcWKC)$ be a weak kernel-cokernel category with enough weak inflations and weak deflations. If $(\mathcal{I}_{1}, \mathcal{J}_{1})$ and $(\mathcal{I}_{2}, \mathcal{J}_{2})$ are complete ideal torsion pairs in $(\mcA; \mcWKC)$, then so are the ideal torsion pairs $(r(\mathcal{J}_{1}\cap\mathcal{J}_{2}), \mathcal{J}_{1}\cap\mathcal{J}_{2})$ and $(\mathcal{I}_{1}\cap\mathcal{I}_{2}, \ell(\mathcal{I}_{1}\cap\mathcal{I}_{2}))$.
\end{coro}

Next assume that $\mcA$ has infinite coproducts. Suppose $\mcM=\{\,f_{s}\colon M_{s}\to N_{s} \,\mid\, s\in\mcS \,\}$ is a set of morphisms. Then for any $s\in\mcS$ and any morphism $x_{s}\colon N_{s}\to X$ in $\mcA$, there is a commutative diagram
$$\xymatrix@R=30pt@C=30pt{
	M_{s}\ar[r]^{f_{s}}\ar[d]_{\iota_{s}} & N_{s}\ar[d]^{e_{s}}\ar[r]^{x_{s}} & X\ar@{=}[d]\\
	\oplus_{s}\,M_{s}\ar[r]^{\oplus_{s}\,f_{s}} & \oplus_{s}\,N_{s}\ar[r]^{x} & X,}
$$
where $\oplus_{s}\,f_{s}\colon \oplus_{s}\,M_{s}\to \oplus_{s}\,N_{s}$ and $x\colon \oplus_{s}\, N_{s}\to X$ are canonical morphisms. Let $g\colon X\to Y$ be an arbitrary morphism in $\mcA$. The universal property of coproducts implies that $gx_{s}f_{s}=0$ for any $s\in\mcS$ if and only if $gx(\oplus_{s}\, f_{s})=0$. It follows that
$(\oplus_{s}\, f_{s})^{\perp} \subseteq {\mcM^{\perp}}$. Conversely, if $g\colon X\to Y$ belongs to $\mcM^{\perp}$, then for any morphism $x'\colon \oplus_{s}\, N_{s}\to X$ in $\mcA$, we have that $gx'(\oplus_{s}\, f_{s}){\iota_{s}}=gx'e_{s}f_{s}=0$ for every $s\in\mcS$, and so $gx'(\oplus_{s}\, f_{s})=0$. It follows that $\mcM^{\perp} \subseteq {(\oplus_{s}\, f_{s})}^{\perp}$, and therefore $\mcM^{\perp}={(\oplus_{s}\, f_{s})}^{\perp}$.

\begin{thm}{\em (Eklof-Trlifaj Lemma)}\label{T: ET}
Let $(\mcA; \mcWKC)$ be a weak kernel-cokernel category with enough weak deflations. If $\mcA$ has infinite coproducts, then the ideal torsion pair $(^{\perp}(\mcM^{\perp}), \mcM^{\perp})$ generated by a set $\mcM$ of morphisms is complete.
\end{thm}

\begin{proof}
We may reduce the argument to the situation that $\mcM$ only contains a single morphism $f\colon A\to B$. Let $X$ be an object in $\mcA$, and $\mcS$ be the set $\Hom(B, X)$. Then for any morphism $s\colon B\to X$, there is a commutative diagram
$$\xymatrix@R=30pt@C=30pt{
		A\ar[d]_{\sigma_{s}}\ar[r]^{f}   & B\ar[r]^{s}\ar[d]^{\sigma_{s}}   & X\ar@{=}[d]\\
		A^{(\mcS)}\ar[r]^{f^{(\mcS)}}    & B^{(\mcS)}\ar[r]^{\delta}        & X}
$$
with $f^{(\mcS)}\colon A^{(\mcS)}\to B^{(\mcS)}$ and $\delta\colon B^{(\mcS)}\to X$ canonical morphisms. Take a weak cokernel $j\colon X\to F$ of the morphism $\delta f^{(\mcS)}\colon A^{(\mcS)}\to X$, which is a weak deflation. One can easily check that $j$ is a $\Hom$-special $f^{\perp}$-preenvelope of $X$, since a morphism $g\colon X\to Y$ belongs to $f^{\perp}$ if and only if $g\delta f^{(\mcS)}=0$. Therefore the ideal torsion pair generated by a morphism $f$ is complete by Corollary \ref{C: Hom special}.
\end{proof}

\section{Weak Exact Categories} \label{S: wec}
In this section, we introduce the notion of a weak exact structure on an additive category and examine the properties of complete ideal torsion pairs in weak exact categories.

\begin{df}
A {\em weak exact category} $(\mcA; \mcWE)$ is a weak kernel-cokernel category satisfying the additional axioms:\vspace{.3em}

$(\WE_{2})$ weak inflations are closed under composition; \vspace{.3em}

$(\WE_{2})^{\rm{op}}$ weak deflations are closed under composition; and \vspace{.3em}

$(\WE_{3})$ {\em (Five Lemma)} for a commutative diagram with each row a weak conflation:
$$\xymatrix@R=30pt@C=30pt{
	X\ar[r]^{k} \ar[d]^f & Y\ar[r]^{c}\ar[d]^g  & Z\ar[d]^h\\
    X'\ar[r]^{k'}        & Y'\ar[r]^{c'}        & Z',}
$$
if two of the three morphisms $f, g$, and $h$ are isomorphisms, then so is the third one.
\end{df}

\begin{exa}\label{E: extri cat}\rm
An extriangulated category $(\mathscr{C},\mathbb{E},\mathfrak{s})$ is an additive category $\mathscr{C}$ equipped with a bifunctor $\mathbb{E}\colon \mathscr{C}^{\textrm{op}}\times\mathscr{C}\to \Ab$ and an $\mathbb{E}$-triangulation $\mathfrak{s}$ satisfying $\rm(ET1)$, $\rm(ET2)$, $\rm(ET3)$, $\rm(ET3)^{\rm{op}}$, $\rm(ET4)$, and $\rm(ET4)^{\rm{op}}$, see \cite[Definition 2.12]{NP}. For any pair $(Z, X)$ of objects in $\mathscr{C}$, an element $\delta \in\mathbb{E}(Z, X)$ is called an $\mathbb{E}$-extension, and $\mathfrak{s}$ associates an equivalence class $\mathfrak{s}(\delta)=[X\stackrel{k}{\longrightarrow}Y\stackrel{c}{\longrightarrow}Z]$ to any $\mathbb{E}$-extension $\delta \in\mathbb{E}(Z, X)$. In this case, one says that the sequence $X\stackrel{k}{\longrightarrow}Y\stackrel{c}{\longrightarrow}Z$ realizes $\delta$. A sequence
$$\xymatrix@R=30pt@C=30pt{
    X\ar[r]^{k}  & Y\ar[r]^{c}  & Z, }
$$
is called a {\em conflation}, if it realizes some $\mathbb{E}$-extension
$\delta\in\mathbb{E}(Z, X)$. A morphism $k\colon X\to Y$ is called an {\em inflation} if it admits some conflation; {\em deflations} are defined dually. It follows from \cite[Proposition 3.3]{NP} that conflations in $(\mathscr{C},\mathbb{E},\mathfrak{s})$ are weak kernel-cokernel pairs, so that we can consider the collection $\mcW_{\extri}$ of all conflations in $(\mathscr{C},\mathbb{E},\mathfrak{s})$, which is closed under isomorphism by \cite[Proposition 3.7]{NP}. Condition $\rm(ET2)$ tells that split $\mathbb{E}$-extensions are conflations and conflations are closed under finite coproducts, and so $\mcW_{\extri}$ satisfies $\rm(WE_0)$, $\rm(WE_1)$, and $\rm(WE_1)^{\rm{op}}$. Conditions $\rm(ET4)$ and $\rm(ET4)^{\rm{op}}$ ensure that inflations and deflations are closed under composition respectively, and hence $\rm(WE_2)$ and $\rm(WE_2)^{\rm{op}}$ hold. But $\rm(WE_3)$ follows from \cite[Corollary 3.6]{NP}. Thus $(\mathscr{C};\mcW_{\extri})$ is a weak exact category.
\end{exa}

\begin{rem}\rm
We do not know of any examples of weak exact categories that are not extriangulated. In a triangulated category, every morphism is part of a triangle and is therefore both a weak inflation as well as a weak deflation. We do not know of any example of a weak exact category all of whose morphisms are both weak inflations and weak deflations that is not triangulated.
\end{rem}

\textbf{In the remainder of this section, we will assume that $(\mcA; \mcWE)$ is a weak exact category}.

\subsection{Elementary properties}
The next proposition follows from $\rm(WE_2)$ immediately.

\begin{prop}\label{P: comm dia in wec}
For two composable weak inflations $k_{1}\colon X\to Y_{1}$ and $k_{2}\colon Y_{1}\to Y_{2}$ in $(\mcA; \mcWE)$, there is a commutative diagram
$$\xymatrix@R=30pt@C=30pt{
   X\ar[r]^{k_{1}}\ar@{=}[d]  & Y_{1}\ar[d]^{k_{2}}\ar[r]^{c_{1}} & Z_{1}\ar@{-->}[d]^{\alpha}\\
   X\ar[r]^{k_{2}k_{1}}       & Y_{2}\ar[d]^{c_{2}}\ar[r]^{c}     & Z\ar@{-->}[d]^{\beta} \\
                              & Z_{2}\ar@{=}[r]                   & Z_{2}, }
$$
where the first and second rows and the second column are weak conflations.
\end{prop}

Let us note that the existence of weak kernels is related to the existence of weak pullbacks. Recall the definition of a {\em weak pullback} of two morphisms $b\colon B\to D$ and
$g\colon C\to D$ in an additive category $\mcA$. It is a pair of morphisms
$a\colon A\to C$ and $f\colon A\to B$ satisfying (1) $bf=ga$; and (2) if given any other pair of morphisms $a'\colon A'\to C$ and $f'\colon A'\to B$ for which $bf'=ga'$, then there is a morphism $\alpha\colon A'\to A$ making the following diagram commute
$$\xymatrix@R=30pt@C=30pt{
    A'\ar@/_1pc/[rdd]_{a'}\ar@/^1pc/[rrd]^{f'}\ar@{-->}[rd]^(.6){\alpha} && \\
      & A\ar[d]_{a}\ar[r]^{f}  & B\ar[d]^{b}\\
      & C\ar[r]^{g}            & D. }
$$
A {\em pullback} of $b$ and $g$ is a weak pullback for which the factorization is unique. A {\em $($weak$)$ pushout} of two morphisms starting in the same object is defined dually. If the morphism $(b,-g)\colon B\oplus C \to D$ has a weak kernel $\theta \colon A \to B\oplus C$, then any morphism
$\left(\begin{smallmatrix}f'\\a'\end{smallmatrix}\right)\colon A' \to B\oplus C$ satisfying $(b,-g)\left(\begin{smallmatrix}f'\\a'\end{smallmatrix}\right)=bf'-ga'=0$ factors through $\theta$ by its weak kernel property, and therefore the pair of morphisms $(0,1)\theta$ and $(1,0)\theta$ is a weak pullback of $b$ and $g$. This discussion may be used to verify the following.

\begin{prop}\label{P: weak pb in wec}
A weak pullback of a weak deflation in $(\mcA; \mcWE)$ along any morphism exists. Dually, a weak pushout of a weak inflation in $(\mcA; \mcWE)$ along any morphism exists.
\end{prop}
\begin{proof}
It sufficient to show that if $c\colon Y\to Z$ is a weak deflation, then so is
$(c',-c)\colon Y'\oplus Y\to Z$ for every morphism $c'\colon Y'\to Z$. But the morphism
$(c',-c)\colon Y'\oplus Y\to Z$ is a composition of the following morphisms
$$\xymatrix@C=30pt{
    Y'\oplus Y\ar[r]^(.5){\left(\begin{smallmatrix}1&~0 \\ 0&-c\end{smallmatrix}\right)}
     & Y'\oplus Z\ar[r]^(.45){\left(\begin{smallmatrix}1&0 \\ c'&1\end{smallmatrix}\right)}
       & Y'\oplus Z\ar[r]^(.55){(0,\,1)} & Z, }
  $$
which are all weak deflations by $\rm(WE_0)$ and $\rm(WE_1)^{\textrm{op}}$. It follows from $\rm(WE_2)^{\textrm{op}}$ that the composition \linebreak $(c',-c)\colon Y'\oplus Y\to Z$ is a weak deflation.
\end{proof}

This proposition implies that there is a commutative diagram
$$\xymatrix@R=30pt@C=30pt{
       X \ar[r]^{f}\ar@{=}[d]    & A \ar[r]^{a'}\ar[d]_{a}   & Y' \ar[d]^{c'} \\
       X \ar[r]^{k}\ar[d]_{k'}   & Y \ar[r]^{c}\ar[d]_{b}    & Z \ar@{=}[d] \\
       X'\ar[r]^{b'}             & B \ar[r]^{g}              & Z }
$$
with the middle row a weak conflation, $a'f = 0$ and $gb'=0$. That is why we call the structure $\mcWE$ weak exact.

\subsection{Minimal Salce's Lemma}
Let $\mcI$ be an ideal of $\mcA$, and $X\in\mcA$ an object. An $\mcI$-precover $i\colon T\to X$ is called an $\mcI$-{\em cover} if any endomorphism $f\colon T\to T$ completing the commutative diagram
$$\xymatrix@R=30pt@C=30pt{
	              & T \ar[d]^{i} \ar@{-->}[ld]_{f} \\
	 T \ar[r]^{i} & X }
$$
is an isomorphism. An ideal $\mcI$ is said to be {\em covering} if every object has an $\mcI$-cover; {\em enveloping ideals} are defined dually.

\begin{thm}{\em (Minimal Salce's lemma)}\label{T: min SL}
Let $(\mcI,\mcJ)$ be a complete ideal torsion pair in $(\mcA; \mcWE)$. Then $\mcI$ is covering if and only if $\mcJ$ is enveloping.
\end{thm}

\begin{proof}
Let $X$ be an object in $\mcA$, and consider a weak conflation
$$\xymatrix@R=30pt@C=30pt{
		T\ar[r]^{i} & X\ar[r]^{j} & F }
$$
with $i\colon T\to X$ an $\mcI$-cover and $j\colon X\to F$ a $\mcJ$-preenvelope of $X$. We need to show that $j$ is a $\mcJ$-envelope of $X$. Consider the following commutative diagram
$$\xymatrix@R=30pt@C=30pt{
		T\ar[r]^{i} & X\ar[r]^{j}\ar@{=}[d]   & F\ar@{-->}[d]^{f}\\
		            & X\ar[r]^{j}             & F, }
$$
where the morphism $f$ is obtained from the fact that $j\colon X\to F$ is a $\mcJ$-preenvelope of $X$. Completing this commutative diagram to the following
$$\xymatrix@R=30pt@C=30pt{
		T\ar[d]_{g}\ar[r]^{i} & X\ar[r]^{j}\ar@{=}[d] & F\ar[d]^{f}\\
		T\ar[r]^{i}           & X\ar[r]^{j}           & F.}
$$
The morphism $g\colon T\to T$ is an isomorphism since $i\colon T\to X$ is an $\mcI$-cover of $X$, and hence $f$ is an automorphism of $F$ by $\rm(WE_3)$. A dual argument shows that if $\mcJ$ is an enveloping ideal, then $\mcI$ is covering.
\end{proof}

An ideal torsion pair $(\mcI,\mcJ)$ is said to be {\em perfect} if $\mcI$ is covering and $\mcJ$ is enveloping.

\subsection{Christensen's Lemma}

For weak exact categories, Theorem \ref{T: prod perp equ} can be strengthened as follows. It tells us that the Galois connection $(\ell, r)$ between $\Hom$-special precovering ideals and $\Hom$-special preenveloping ideals replaces products with weak extensions of ideals and conversely. It is a $0$-dimensional analogue of \cite[Theorem 8.4]{FH} on $\Ext$-orthogonality of $\Ext$-special approximating ideals.

\begin{thm}{\em (Christensen's Lemma)}\label{T: FH in wec}
If $(\mcI_{1}, \mcJ_{1})$ and $(\mcI_{2}, \mcJ_{2})$ are two complete ideal torsion pairs in $(\mcA; \mcWE)$, then so are $(\mcI_{1}\mcI_{2}, {\mcJ_{2}}\diamond{\mcJ_{1}})$ and $({\mcI_{1}}\diamond{\mcI_{2}}, \mcJ_{2}\mcJ_{1})$.
\end{thm}

\begin{proof}
We only prove that $(\mcI_1\mcI_2, \mcJ_2\diamond\mcJ_1)$ is a complete ideal torsion pair, by showing that $\Hom$-special precovering ideals are closed under the product operation. Indeed, if $\mcI_1$ and $\mcI_2$ are $\Hom$-special precovering ideals and $X \in \mcA$ with a $\Hom$-special $\mcI_1$-precover $i_1 \colon T_1 \to X$, let $i_2 \colon T_2 \to T_1$ be a $\Hom$-special $\mcI_2$-precover of $T_1$, then $i_1 i_2 \colon T_2 \to X$ is an $\mcI_1 \mcI_2$-precover, which is a weak inflation, by Proposition~\ref{P: sum and prod} and $\rm(WE_2)$. It follows that $\mcI_1 \mcI_2$ is a $\Hom$-special precovering ideal.
\end{proof}

An ideal $\mcI$ is {\em closed under weak extensions} if $\mcI\diamond\mcI=\mcI$, and $\mcI$ is an {\em idempotent} ideal if $\mcI^2=\mcI$. The next result follows from Theorem \ref{T: FH in wec} immediately.

\begin{coro}\label{C: ide and clo ext}
Let $(\mcI, \mcJ)$ be a complete ideal torsion pair in $(\mcA; \mcWE)$. The ideal $\mcI$ is idempotent if and only if $\mcJ$ is closed under weak extensions; the ideal $\mcI$ is closed under weak extensions if and only if $\mcJ$ is idempotent.
\end{coro}

Note that by \cite[Prposition 10.5]{FH}\label{lemma}, an idempotent covering ideal is an object ideal, and dually,  an idempotent enveloping ideal is an object ideal. If $(\mcI,\mcJ)$ is a complete ideal torsion pair in $(\mcA; \mcWE)$ with $\mcI$ a covering ideal closed under weak extensions, then $\mcJ$ is an idempotent enveloping ideal by Theorem \ref{T: min SL} and Corollary \ref {C: ide and clo ext}, and hence an object ideal. This proves the following result.

\begin{thm}\label{P: min ob ide} {\em (Wakamatsu's Lemma)}
Let $(\mcI,\mcJ)$ be a complete ideal torsion pair in $(\mcA; \mcWE)$. If $\mcI$ is a covering ideal closed under weak extensions, then $\mcJ$ is an object ideal.
\end{thm}

\section{Complete Ideal Torsion Pairs in Abelian Categories}\label{S: Ab}
\textbf{In this section, we assume that $\mcA$ is an abelian category}. Then the collection of short exact sequences forms a weak exact structure on $\mcA$. Note that for a precovering ideal $\mcI$ in $\mcA$ and an object $X\in\mcA$, there is a short exact sequence
$$\xymatrix@R=30pt@C=30pt{
   0\ar[r]  & T \ar[r]^{i}  & X \ar[r]^{j}  & F \ar[r]  & 0, }
$$
where $i\colon T\to X$ is an $r(\ell(\mcI))$-cover and $j\colon X\to F$ an $\ell(\mcI)$-envelope of $X$. These considerations may be documented as follows.

\begin{thm}\label{T: SL in ab}
Let $\mcA$ be an abelian category. The following are equivalent for a precovering ideal $\mcI$ in $\mcA:$
\begin{enumerate}
\item $r(\ell(\mcI))=\mcI;$
\item every object $X\in\mcA$ has a monic $\mcI$-cover; and
\item the \Hom-orthogonal pair $(\mcI, \ell(\mcI))$ is a complete ideal torsion pair.
\end{enumerate}
\end{thm}

\subsection{Complete ideal torsion pairs and preradicals}

It can be easily observed that every complete ideal torsion pair in $\mcA$ is perfect. Let $(\mcI, \mcJ)$ be a complete ideal torsion pair in $\mcA$. The rule that assigns to every object $X\in\mcA$ the attending short exact sequence
$$\xymatrix@R=30pt@C=30pt{
     0\ar[r] & T(X)\ar[r]^(.55){i_{X}} & X \ar[r]^(.45){j_{X}} & F(X)\ar[r] &0 }
$$
with $i_{X}\in\mcI$ and $j_{X}\in\mcJ$, is functorial.  If $f\colon X\to Y$ is a morphism in $\mcA$, then there arises a morphism of the associated short exact sequences
$$\xymatrix@R=30pt@C=30pt{
     0\ar[r] & T(X)\ar[r]^(.55){i_{X}}\ar@{-->}[d]_{T(f)} & X\ar[r]^(.45){j_{X}}\ar[d]^{f}
                                            & F(X)\ar[r]\ar@{-->}[d]^{F(f)} & 0 \\
     0\ar[r] & T(Y)\ar[r]^(.55){i_{Y}}                    & Y\ar[r]^(.45){j_{Y}}
                                            & F(Y)\ar[r]                    & 0.}
$$
The morphism $T(f)\colon T(X)\to T(Y)$ exists, because $fi_{X}\colon T(X)\to Y$ belongs to the ideal $\mcI$ and $i_{Y}\colon T(Y)\to Y$ is an $\mcI$-cover of $Y$. Furthermore, $T(f)$ is the only morphism in $\Hom(T(X), T(Y))$ that makes the diagram commute, that is why it can be denoted as a function on $f$. The rule $X \mapsto T(X)$, $f \mapsto T(f)$, is also functorial, denoted by $T=T(\mcI)\colon \mcA\to \mcA$, and the morphisms $i_{X}\colon T(X)\to X$ constitute a natural transformation $i\colon T\to\Id$, where $\Id\colon \mcA\to \mcA$ denotes the identity functor on $\mcA$. The above argument also implies that for an object $X$ in $\mcA$, the associated short exact sequence is unique up to isomorphism of short exact sequences.

\begin{df}\label{D: preradical}
{\em (\cite[Chapter \uppercase\expandafter{\romannumeral6}. \S 1 and \S 2]{S})} A {\em preradical} of an abelian category $\mcA$ is an additive functor
$t\colon \mcA\to\mcA$ together with a natural transformation $\iota\colon t\to \Id$ such that for every object $X\in\mcA$, the associated morphism $\iota_{X}\colon t(X)\to X$ is a monomorphism in $\mcA$. The monomorphism $\iota_{X}$ is called the {\em torsion subobject} of $X$.
\end{df}

We have thus associated to a complete ideal torsion pair $(\mcI, \mcJ)$ in $\mcA$ the preradical $T=T(\mcI)$ of $\mcA$. A morphism $f\colon X\to Y$ belongs to the ideal $\mcI$ if and only if it factors through the torsion subobject $i_{Y}\colon T(Y)\to Y$ of the codomain, so we may infer that $\mcI$ is the ideal generated by the torsion subobjects $i_{X}$ of $X\in\mcA$.

\begin{thm}\label{T: preradical}
Let $\mcA$ be an abelian category. The rule
$(\mcI, \mcJ) \mapsto T(\mcI)$ is a bijective correspondence between complete ideal torsion pairs in $\mcA$ and preradicals of $\mcA$.
\end{thm}
\begin{proof} Let $t$ be a preradical of $\mcA$, and denote by $\mcI(t)$ the ideal generated by the torsion subobjects $\iota_{X}\colon t(X)\to X$ of $X\in \mcA$. We will prove that the $\Hom$-orthogonal pair $(\mcI(t), \ell(\mcI(t)))$ is a complete ideal torsion pair and that the preradical $T(\mcI(t))$ associated to it is $t$ itself. To that end, it suffices to verify that for every object $X\in\mcA$, the torsion subobject $\iota_{X}\colon t(X)\to X$ is an $\mcI(t)$-cover of $X$. Theorem \ref{T: SL in ab} then implies that $(\mcI(t), \ell(\mcI(t)))$ is a complete ideal torsion pair, and the foregoing considerations imply that the preradical associated to it is given by $t$.

Consider the collection $\mcI'(t)$ of those morphisms $f\colon X\to Y$ in $\mcA$ that factor through the torsion subobject $\iota_{Y}\colon t(Y)\to Y$ of the codomain. Evidently, $\mcI'(t) \subseteq \mcI(t)$ and $\mcI'(t)$ contains every torsion subobject $\iota_{X}\colon t(X)\to X$ of $X\in\mcA$. But $\mcI'(t)$ is itself an ideal, for it is clearly closed under the sum operation and composition on the right by an arbitrary morphism in $\mcA$, and to prove that it is closed under composition from the left, suppose that $f\colon X\to Y$ belongs to $\mcI'(t)$ and that $g\colon Y\to Z$ is an arbitrary morphism in $\mcA$. The following commutative diagram arises
$$\xymatrix@R=30pt@C=30pt{
                                          & X\ar[ld]\ar[d]^{f}\\
      t(Y)\ar[r]^{\iota_{Y}}\ar[d]_{t(g)} & Y \ar[d]^{g} \\
      t(Z)\ar[r]^{\iota_{Z}}              & Z }
$$
and indicates how the composition $gf$ too factors through the torsion subobject of the codomain.
It follows that $\mcI'(t)=\mcI(t)$. By the definition of $\mcI'(t)$, the torsion subobject $\iota_{X}\colon t(X)\to X$ is an $\mcI(t)$-cover of $X$ in $\mcA$.
\end{proof}

According to the proof of Theorem \ref{T: preradical}, the associated preradical of a complete ideal torsion pair $(\mcI, \mcJ)$ is some $t$ if and only if the domain of the $\mcI$-cover $i_{X}\colon T(X)\to X$ is isomorphic to $t(X)$ for every object $X\in\mcA$. Let $t_{1}$ and $t_{2}$ be two preradicals of $\mcA$. One may form the preradicals $t_{1}+t_{2}$, $t_{1} \cap t_{2}$ and $t_{1}t_{2}$ by assigning to each object $X\in\mcA$ the subobjects $t_{1}(X)+t_{2}(X)$, $t_{1}(X) \cap t_{2}(X)$ and $t_{1}(t_{2}(X))$, respectively. Recall that one defines the preradical $t_{1}\colon\!t_{2}$ as
$$(t_{1}\colon\!t_{2})(X)/t_{1}(X)=t_{2}(X/t_{1}(X)),$$
for every object $X$ in $\mcA$ \cite[Chapter \uppercase\expandafter{\romannumeral6}. \S 1]{S}. Combining with Corollary \ref{C: int and sum in wac} and Theorem \ref{T: FH in wec}, we have the following.

\begin{thm}\label{P: preradicals}
Let $(\mcI_{1}, \mcJ_{1})$ and $(\mcI_{2}, \mcJ_{2})$ be two complete ideal torsion pairs in $\mcA$, and let $t_{1}$ and $t_{2}$ be the preradicals associated to them respectively. Then
\begin{enumerate}
\item $t_{1}+t_{2}$ is the associated preradical of the compete ideal torsion pair $(r(\mcJ_{1}\cap\mcJ_{2}), \mcJ_{1}\cap\mcJ_{2});$
\item $t_{1} \cap t_{2}$ is the associated preradical of the complete ideal torsion pair $(\mcI_{1}\cap\mcI_{2},\ell(\mcI_{1}\cap\mcI_{2}));$
\item $t_{2}t_{1}$ is the associated preradical of the complete ideal torsion pair $({\mcI_{1}}{\mcI_{2}},{\mcJ_{2}}\diamond{\mcJ_{1}});$ and
\item  $t_{1}\colon\!t_{2}$ is the associated preradical of the complete ideal torsion pair $(\mcI_{1}\diamond\mcI_{2}, \mcJ_{2}\mcJ_{1})$.
\end{enumerate}
\end{thm}
\begin{proof}
For an object $X\in\mcA$ and $s=1,2$, we let
$$\xymatrix@R=30pt@C=30pt{
     0\ar[r] & T_{s}\ar[r]^{i_{s}} & X \ar[r]^{j_{s}} & F_{s}\ar[r] & 0 }
$$
be the associated short exact sequence of $X$ related to the complete ideal torsion pair $(\mcI_{s}, \mcJ_{s})$, which is to say that $t_{s}(X)=T_{s}$. \vspace{.3em}

$(1)$ The canonical morphism $(i_{1},i_{2})\colon  T_{1}\oplus T_{2} \to X$ then induces a right exact sequence
$$\xymatrix@R=30pt@C=30pt{
      T_{1}\oplus T_{2}\ar[r]^(.6){(i_{1},\,i_{2})} & X \ar[r]^{j} & F\ar[r]  & 0, }
$$
where $j\colon X\to F$ is a $\mcJ_{1}\cap\mcJ_{2}$-envelope of $X$. This gives rise to a short exact sequence
$$\xymatrix@R=30pt@C=30pt{
     0\ar[r] & T_{1}+T_{2}\ar[r]^(.6){i} & X \ar[r]^{j} & F\ar[r] & 0 }
$$
related to the complete ideal torsion pair $(r(\mcJ_{1}\cap\mcJ_{2}), \mcJ_{1}\cap\mcJ_{2})$. Since $T_{1}+T_{2}=(t_{1}+t_{2})(X)$, we conclude that $t_{1}+t_{2}$ is the associated preradical of $(r(\mcJ_{1}\cap\mcJ_{2}), \mcJ_{1}\cap\mcJ_{2})$. \vspace{.3em}

$(2)$ The associated short exact sequence of $X$, which is related to the complete ideal torsion pair $(\mcI_{1}\cap\mcI_{2}, \ell(\mcI_{1}\cap\mcI_{2}))$, is obtained by taking the cokernel of the kernel of
$\left(\begin{smallmatrix}j_{1}\\j_{2}\end{smallmatrix}\right)\colon X\to F_{1}\oplus F_{2}$,
and so is of the form
$$\xymatrix@R=30pt@C=30pt{
     0\ar[r] & T_{1}\cap T_{2}\ar[r]^(.6){i'} & X \ar[r]^{j'} & F'\ar[r] & 0. }
$$
It follows that the associated preradical of $(\mcI_{1}\cap\mcI_{2},
\ell(\mcI_{1}\cap\mcI_{2}))$ is $t_{1}\cap t_{2}$. \vspace{.3em}

$(3)$ Trivial.\vspace{.3em}

$(4)$ Let $t$ be the associated preradical of $(\mcI_{1}\diamond\mcI_{2}, \mcJ_{2}\mcJ_{1})$. For the object $F_{1}$, there is a short exact sequence
$$\xymatrix@R=30pt@C=30pt{
     0\ar[r] & T'_{2}\ar[r]^{i'_{2}} & F_{1} \ar[r]^{j'_{2}} & F'_{2}\ar[r] & 0, }
$$
which is related to $(\mcI_{2}, \mcJ_{2})$, and so $T'_{2}=t_{2}(F_{1})=t_{2}(X/t_{1}(X))$. Then we have the following commutative diagram
$$\xymatrix@R=30pt@C=30pt{
              &                        & 0\ar[d]      & 0\ar[d]\\
      0\ar[r] & T_{1}\ar[r]\ar@{=}[d]  & T\ar[d]^{h}\ar[r]
       & T'_{2}\ar[r]\ar[d]^{i'_{2}}                  & 0 \\
      0\ar[r] & T_{1}\ar[r]^{i_{1}}    & X\ar[d]^{j'_{2}j_{1}}\ar[r]^{j_{1}}
       & F_{1}\ar[r]\ar[d]^{j'_{2}}                   & 0 \\
              &                        & F'_{2}\ar[d]\ar@{=}[r]
       & F'_{2}\ar[d]\\
              &                        & 0            & 0 }
$$
with exact rows and columns. Note that the second column is the associated short exact sequence of $X$ related to the complete ideal torsion pair $(\mcI_{1}\diamond\mcI_{2}, \mcJ_{2} \mcJ_{1})$. It follows that $t(X)=T$, and hence $t(X)/t_{1}(X)=T'_{2}=t_{2}(X/t_{1}(X))$. Therefore $t=t_{1}\colon\!t_{2}$.
\end{proof}

Let $(\mcI,\mcJ)$ be a complete ideal torsion pair in $\mcA$, and $X$ an object of $\mcI$. In that case, the identity morphism $1_{X}:X\to X$ is an $\mcI$-cover of $X$. It means that if $t$ is the associated preradical, then the torsion subobject $\iota_{X}:t(X)\to X$ of $X$ is given by $1_{X}$. Thus $t(X)=X$ and we conclude that the objects of $\mcI$ are the {\em torsion objects} \cite[Chapter \uppercase\expandafter{\romannumeral6}. \S 2]{S} of $\mcA$ associated to the preradical $t$. Dually, suppose that $Y$ is an object of the ideal $\mcJ$. Then $1_{Y}\colon Y\to Y$ is a $\mcJ$-envelope of $Y$ and the torsion subobject $\iota_{Y}\colon t(Y)\to Y$ of $Y$ appears in the short exact sequence
$$\xymatrix@R=30pt@C=30pt{
     0\ar[r] & t(Y)\ar[r]^(.55){\iota_{Y}} & Y \ar[r]^{1_{Y}} & Y\ar[r] & 0 }
$$
as its kernel. Equivalently, $t(Y)=0$ and the objects of $\mcJ$ are the {\em torsion-free objects} (loc. cit.) of $\mcA$ associated to the preradical $t$.

A preradical $t$ is {\em idempotent} if $tt=t$, and is called a {\em radical} if $t\colon\!t=t$ \cite[Chapter \uppercase\expandafter{\romannumeral6}. \S 1]{S}. Let $t$ be a preradical of $\mcA$ and $(\mcI, \mcJ)$ the corresponding complete ideal torsion pair. Then $t$ is idempotent if and only if $\mcI$ is an idempotent ideal, and $t$ is a radical if and only if $\mcI$ is closed under weak extensions. The next result follows from \cite[Prposition 10.5]{FH} and Corollary \ref{C: ide and clo ext} directly, which implies that for a complete ideal torsion pair $(\mcI, \mcJ)$ in $\mcA$, both $\mcI$ and $\mcJ$ are object ideals if and only if the associated preradical $t$ is an idempotent radical. In that case, the pair $(\Ob(\mcI), \Ob(\mcJ))$ of subcategories constitutes a torsion theory for $\mcA$ in the classical sense of \cite[Chapter \uppercase\expandafter{\romannumeral6}. \S 2]{S}.

\begin{prop}\label{P: rad}
Let $(\mcI, \mcJ)$ be a complete ideal torsion pair in $\mcA$, and $t$ the associated preradical. The ideal $\mcI$ is an object ideal if and only if $t$ is idempotent$\,;$ the ideal $\mcJ$ is an object ideal if and only if $t$ is a radical.
\end{prop}

\subsection{Image maximal ideals}\label{S:im ideals}
Note that in an abelian category $\mcA$, every morphism $f\colon Y\to X$ factors through its image as $Y\stackrel{}{\to}\im f \stackrel{} {\rightarrow}X$. For an ideal $\mcI$ in $\mcA$ and $X\in\mcA$, we consider the system of all subobjects of the form $\im f$, with $f$ a morphism in $\mathcal{I}$ ending in $X$. It is directed, since any two morphisms $f,g \in \mathcal{I}$ ending in $X$ induce a morphism $h\in \mathcal{I}$ with $\im h = \im f + \im g$.

\begin{df} \label{D: im max}
An ideal $\mcI$ is {\em image maximal} if for every object $X\in\mcA$, the system $\im f$ with $f\in \mathcal{I}$ ending in $X$ has a maximal element.
\end{df}

For example, if $\mcI$ is a precovering ideal, then $\mcI$ is image maximal as a maximal element equals the image of an $\mcI$-precover of $X$. Hence the following proposition may be seen as a strengthening of Proposition \ref{P: precover complete} for abelian categories.

\begin{prop}\label{P: im max complete}
If $\mcI$ is an image maximal ideal in $\mcA$, then the ideal $r(\ell(\mcI))$ is monic covering in $\mcA$. Further, an $r(\ell(\mcI))$-cover of $X\in\mcA$ is given by an inclusion $\iota\colon T\to X$ with $T=\im f$ and $f\in\mcI(-,X)$.
\end{prop}
\begin{proof}
Consider the ideal torsion pair $(r(\ell(\mathcal{I})), \ell(\mathcal{I}))$ generated by $\mathcal{I}$. For an object $X\in\mathcal{A}$, let $T$ be the maximal subobject of $X$ with $T=\im f$ and $f\in\mcI(-,X)$. Because $r(\ell(\mathcal{I}))$ is a right annihilator ideal, it follows that the inclusion $\iota\colon T\to X$ belongs to $r(\ell(\mathcal{I}))$. Now the epimorphism $\pi\colon X \to X/T$ belongs to $\ell(\mathcal{I})$ as $\pi f=0$ for all $f\in\mcI(-,X)$. Thus if $g\colon  M\to X$ belongs to $r(\ell(\mathcal{I}))$, then $\pi g=0$ and $\im g \subseteq T$. It follows that $\iota\colon T\to X$ is a monic $r(\ell(\mathcal{I}))$-cover of $X$. Theorem \ref{T: SL in ab} then implies that $(r(\ell(\mathcal{I})), \ell(\mathcal{I}))$ is complete.
\end{proof}

In many cases the abelian category $\mcA$ fulfills the {\em subobject condition}, that is, for every object $X\in\mcA$, the directed system of subobjects of $X$ has a maximal element. Clearly, in that case every ideal $\mcI$ in $\mcA$ is image maximal.

\begin{coro}\label{C: subobj comp}
If $\mcA$ fulfills the subobject condition, then every ideal torsion pair in $\mcA$ is complete.
\end{coro}

The next result is the analogue of \cite[Theorem 2]{FGHT} for ideal torsion pairs in an abelian category.

\begin{prop}\label{P: ob in ab}
Let $\mcI$ be an image maximal ideal in $\mcA$. If $\mcI$ is an object ideal, then so is $r(\ell(\mcI))$. Further, the subcategory of objects of $r(\ell(\mcI))$ is the closure of $\Ob(\mcI)$ under quotient objects.
\end{prop}
\begin{proof}
Let $X$ be an object in $\mcI$ and $\pi\colon X\to Y$ an epimorphism.  Then $1_{X}\in\mcI$ implies $\pi\in\mcI \subseteq r(\ell(\mcI))$ and because $r(\ell(\mcI))$ is a right annihilator, it follows that $1_{Y}\in r(\ell(\mcI))$. Hence the closure of $\Ob(\mcI)$ under quotient objects is contained in the subcategory of objects of $r(\ell(\mcI))$. Now let $f\colon Z\to X$ be a morphism in $r(\ell(\mcI))$. By Proposition \ref{P: im max complete}, there exists a morphism $g\colon T\to X$ in $\mcI$ with $\im f \subseteq \im g$. Since $\mcI$ is an object ideal, there exists an object $T'\in\Ob(\mcI)$ such that $g=g'h$ with $h\colon T\to T'$ and $g'\colon T'\to X$. Hence $\im f \subseteq \im g \subseteq \im g'$ and $f$ factors through the object $\im g'$, which belongs to the closure of $\Ob(\mcI)$ under quotient objects.
\end{proof}

\begin{coro}\label{C: ob in ab}
Let $\mcX$ be an additive subcategory of $\mcA$ which is closed under quotient objects. If $\mcI(\mcX)$ is image maximal, then $(\mcI(\mcX), \ell(\mcI(\mcX))$ is a complete ideal torsion pair.
\end{coro}
\begin{proof}
By Proposition \ref{P: ob in ab}, the ideal $r(\ell(\mcI(\mcX))$ is an object ideal and its collection of objects equals $\mcX$. Hence $(\mcI(\mcX), \ell(\mcI(\mcX))$ is a complete ideal torsion pair by Proposition \ref{P: im max complete}.
\end{proof}

For a complete ideal torsion pair $(\mcI, \mcJ)$ in $\mcA$, we may ask the question: when is the ideal $\mcI$ preenveloping?

\begin{thm}\label{L: proj envelope}
If $\mcI$ be an image maximal ideal in an abelian category $\mcA$ with enough projective objects, then the ideal $r(\ell(\mcI))$ is preenveloping if and only if every projective object in $\mcA$ admits an $\mcI$-preenvelope.
\end{thm}
\begin{proof}
Let $P\in\mcA$ be projective and $f\colon P\to X$ a morphism in $r(\ell(\mcI))$. Then by Proposition \ref{P: im max complete}, there exists a morphism $g\colon Y\to X$ in $\mcI$ with $\im f \subseteq \im g$. Because $P$ is projective, there exists a morphism $h\colon P\to Y$ such that the diagram
$$\xymatrix@R=30pt@C=30pt{
               &                      & P\ar@{-->}[dll]_h\ar[dl]\ar[d]^{f}  \\
      Y \ar[r] & \im g \ar[r]\ar[r]   & X}
$$
commutes. Thus $f=gh\in\mcI$. It follows that every morphism starting in a projective object belongs to $r(\ell(\mcI))$ if and only if it belongs to  $\mcI$. Hence an $r(\ell(\mcI))$-preenvelope of a projective object corresponds to an $\mcI$-preenvelope. In particular, if $r(\ell(\mcI))$ is preenveloping, then every projective object in $\mcA$ admits an $\mcI$-preenvelope.
Now assume that every projective object in $\mcA$ admits an $\mcI$-preenvelope and thus an $r(\ell(\mcI))$-preenvelope. Let $X\in\mcA$ be arbitrary and take an epimorphism $\pi\colon P\to X$ with $P$ projective. Further, let $f\colon P\to Y$ be an $r(\ell(\mcI))$-preenvelope of $P$ and consider the pushout diagram
$$\xymatrix@R=30pt@C=30pt{
      P \ar[r]^\pi \ar[d]_f & X \ar[d]^g \\
      Y \ar[r]              & Z.}
$$
We will show that $g$ is an $r(\ell(\mcI))$-preenvelope of $X$. First,
$f\in r(\ell(\mcI))$ implies $g\pi\in r(\ell(\mcI))$. Since $\pi$ is an epimorphism and $r(\ell(\mcI))$ is a right annihilator ideal, we have
$g\in r(\ell(\mcI))$. Now let $g'\colon X\to Z'$ be an arbitrary morphism in $r(\ell(\mcI))$. Then $g'\pi$ factors through $f$ and by the pushout property, there exists a morphism $Z\to Z'$ such that the following diagram
$$\xymatrix@R=30pt@C=30pt{
     P \ar[r]^\pi \ar[d]_f    & X \ar[d]^g  \ar@/^1pc/[rdd]^{g'} & \\
     Y \ar@/_1pc/[drr] \ar[r] & Z \ar@{-->}[rd]                  & \\
                                                               & & Z' }
$$
commutes. It follows that $g$ is an $r(\ell(\mcI))$-preenvelope of $X$.
\end{proof}

\begin{exa}\label{E: preradical}\rm
Let $R$ be a left noetherian ring and denote by $R$-$\Mod$ the category of left $R$-modules. The work of Enochs \cite{E} shows that the object ideal $\mcI=\mcI(R$-$\Inj)$ consisting of all morphisms that factor through an injective left $R$-module is covering in $R$-$\Mod$. If $M$ is a left $R$-module and $i\colon E\to M$ is an $\mcI$-cover of $M$ with $E$ injective, then the domain of the $r(\ell(\mcI))$-cover of $M$ is given by $\im i$, which is the maximal singular submodule $Z(M)$ of $M$. Proposition \ref{P: ob in ab} implies that $Z(M)$ is an object of $r(\ell(\mcI))$ and $r(\ell(\mcI))$ is the object ideal of all morphisms that factor through a cosyzygy, a quotient of an injective object in $R$-$\Mod$. The associated preradical of the complete ideal torsion pair $(r(\ell(\mcI)), \ell(\mcI))$ generated by $\mcI$ is the left exact preradical associated to the pretopology on $R$ given by the family of essential left ideals of $R$ \cite[Chapter \uppercase\expandafter{\romannumeral6}. \S 6]{S}.

Now consider the category $R$-$\mbox{\rm mod}$ of finitely generated left $R$-modules, which is also an abelian category and further fulfills the subobject condition, so by Corollary \ref{C: subobj comp} every ideal torsion pair in $R$-$\mbox{\rm mod}$ is complete. Corollary \ref{C: ob in ab} implies that every additive subcategory of $R$-$\mbox{\rm mod}$ which is closed under quotient objects yields a complete ideal torsion pair. Further, every preenveloping ideal $\mcJ$ generates a complete ideal torsion pair $(r(\ell(\mcJ)), \ell(\mcJ))$ with $r(\ell(\mcJ))$ preenveloping, in fact, it is enough that $R$ admits a $\mcJ$-preenvelope.
\end{exa}

\section{Frobenius Categories}
Let $(\mcA; \mcE)$ be an exact category, and denote by $\mcE$-$\Proj$ (resp., $\mcE$-$\Inj$) the subcategory of projective objects (resp., injective objects) in $(\mcA;\mcE)$. The exact category $(\mcA; \mcE)$ is said to be {\em Frobenius} \cite{B} if: (1) there exist enough projective objects in $(\mcA; \mcE)$; (2) there exist enough injective objects in $(\mcA; \mcE)$; and (3) $\mcE$-$\Proj=\mcE$-$\Inj$. Let us recall the definition of the {\em stable category}, denoted by $\underline{\mcA}$, of a Frobenius category $(\mcA; \mcE)$ and the attending {\em stable quotient functor} $\pi\colon (\mcA; \mcE)\to \underline{\mcA}$. The objects in $\underline{\mcA}$ are nothing more than the objects $A$ in $\mcA$, only denoted by $\pi(A)$ and referred to as {\em stable objects}. For two objects $A$ and $B$ in $\mcA$, the abelian group of $\underline{\mcA}$-morphisms between the stable objects $\pi(A)$ and $\pi(B)$ is given by
$$\underline{\Hom}(\pi(A), \pi(B)) :=\Hom(A,B)/\Proj(A,B),$$
where $\Proj(A,B)\subseteq\Hom(A,B)$ is the subgroup of $\mcA$-morphisms that factor through a projective object in $(\mcA; \mcE)$. The {\em stable class} of a morphism $f\colon A\to B$ in $\mcA$ is denoted by $\pi(f)\colon \pi(A)\to\pi(B)$, and the rule $\pi\colon (\mcA; \mcE)\to \underline{\mcA}$ that assigns to an object $A$ the stable object $\pi(A)$ and to the morphism $f\colon A\to B$ its stable class $\pi(f)$ is functorial. The stable category $\underline{\mcA}$ is therefore the quotient category of $\mcA$ modulo the ideal of morphisms that factor through $\mcE$-$\Proj= \mcE$-$\Inj$, and so the stable quotient functor $\pi$ is dense and full.

\subsection{Bijective correspondences}
Let $\mcI$ be an ideal of a Frobenius category $(\mcA; \mcE)$. We say that $\mcI$ {\em contains} (the subcategory) $\mcE$-$\Proj=\mcE$-$\Inj$ if for every object $P$ in $\mcE$-$\Proj=\mcE$-$\Inj$, the identity morphism $1_{P}$ belongs to $\mcI(P, P)$. Since the stable quotient functor $\pi$ is dense and full, the image $\pi(\mcI)$ of $\mcI$ is an ideal of $\underline{\mcA}$. These considerations imply the first statement in the following.

\begin{prop}\label{P: in F}
Let $(\mcA; \mcE)$ be a Frobenius category. The stable quotient functor induces a bijective correspondence $\mcI\mapsto \pi(\mcI)$ between the ideals of $\mcA$ that contain $\mcE$-$\Proj=\mcE$-$\Inj$ and the ideals of the stable category $\underline{\mcA}$. Furthermore, an ideal $\mcI$ of $\mcA$ that contains $\mcE$-$\Proj=\mcE$-$\Inj$ is precovering in $\mcA$ if and only if $\pi(\mcI)$ is a precovering ideal in $\underline{\mcA}$. In this case, every object $A\in\mcA$ has an $\mcI$-precover that is a deflation in $(\mcA; \mcE)$.
\end{prop}
\begin{proof}
To prove the second statement, let $A\in\mcA$ and suppose that an $\mcI$-precover $i\colon T\to A$ of $A$ is given. Because the stable quotient functor $\pi$ is dense and full, the stable class $\pi(i)\colon \pi(T)\to\pi(A)$ of this $\mcI$-precover is a $\pi(\mcI)$-precover of $\pi(A)$. To prove the converse, as well as the last statement, let us suppose that a $\pi(\mcI)$-precover of $\pi(A)$ is given. Because the stable quotient functor $\pi$ is dense and full, such a $\pi(\mcI)$-precover is of the form $\pi(i)\colon \pi(T)\to\pi(A)$, where $i\colon T\to A$ belongs to $\mcI$, by the first statement of this proposition. There are enough projective objects in $(\mathcal{A};\mathcal{E})$, so let us take the pullback of the conflation $K\stackrel{}{\rightarrow}P\stackrel{p}{\rightarrow}A$ along $i$, where $P$ belongs to $\mathcal{E}$-$\Proj$. This gives rise to the morphism of conflations depicted by
$$\xymatrix@R=30pt@C=30pt{
     K \ar[r]\ar@{=}[d]  & Q \ar[r]\ar[d]       & T \ar[d]^{i}\\
     K \ar[r]            & P \ar[r]^{p}         & A. }
$$
The dual of \cite[Proposition 2.12]{B} now provides the conflation
$$\xymatrix@R=30pt@C=30pt{
     Q\ar[r]  & T\oplus P\ar[r]^(.6){(i,\,p)}  & A }
$$
whose deflation $(i,p)$ belongs to $\mcI$, because $\mcI\supseteq\mcE$-$\Proj$. To verify that $(i,p)\colon T\oplus P\to A$ is an $\mcI$-precover of $A$, let $i'\colon T'\to A$ be a morphism in $\mcI$. The stable class $\pi(i')\colon \pi(T')\to\pi(A)$ belongs to $\pi(\mcI)$ and so factors through $\pi(i)$ as $\pi(i')=\pi(i)\pi(f)$. The morphism $i'-if\colon T'\to A$ thus factors through a projective object, and therefore through $p$. It follows that there exists a morphism
$g\colon T'\to P$ satisfying $i'-if=pg$, and so the morphism $i'$ factors through $(i,p)$, as required.
\end{proof}

Let $A$ be an object in a Frobenius category $(\mcA;\mcE)$. Since $(\mcA;\mcE)$ has enough injective objects, there is a conflation
$$\xymatrix@R=30pt@C=30pt{
     A \ar[r]  & E \ar[r]  & \Sigma(A), }
$$
where $E$ belongs to $\mcE$-$\Inj$. The stable object $\pi(\Sigma(A))$ of the cosyzygy object of $A$ is well-defined up to unique isomorphism in $\underline{\mcA}$. Moreover, if $f\colon A\to B$ is a morphism in $\mcA$, then a morphism of conflations is induced as in
$$\xymatrix@R=30pt@C=30pt{
     A \ar[r]\ar[d]_{f} & E\ar[r]\ar[d]  & \Sigma(A)\ar[d]^{\Sigma(f)}\\
     B \ar[r]           & E'\ar[r]       & \Sigma(B) }
$$
and the stable class $\pi(\Sigma(f))\colon \pi(\Sigma(A))\to\pi(\Sigma(B))$ of the cosyzygy morphism of $f$ is well-defined up to unique isomorphism in $\underline{\mcA}$. The rule $A\mapsto\pi(\Sigma(A))$, $f\mapsto\pi(\Sigma(f))$, denoted simply by $\Sigma\colon \mcA\to\underline{\mcA}$, is functorial. The functor $\Omega\colon \mcA\to\underline{\mcA}$ is defined in a dual manner. In this case, the functor $\Sigma\colon \mcA\to\underline{\mcA}$ induces an endofunctor of $\underline{\mcA}$, also denoted by $\Sigma\colon \underline{\mcA}\to\underline{\mcA}$, and its inverse is the induced endofunctor $\Omega\colon \underline{\mcA}\to\underline{\mcA}$. The stable category $\underline{\mcA}$ of $(\mcA;\mcE)$ therefore has the structure of a triangulated category as follows: given a morphism $f\colon A\to B$ in $\mcA$, fix a conflation $A\to E\to\Sigma(A)$ in $(\mcA;\mcE)$ with $E$ injective and consider the pushout diagram
$$\xymatrix@R=30pt@C=30pt{
     A \ar[r]\ar[d]_{f}  & E \ar[r]\ar[d]   & \Sigma(A)\ar@{=}[d]\\
     B \ar[r]            & D \ar[r]         & \Sigma(A). }
$$
The induced sequence
$$\xymatrix@R=30pt@C=30pt{
      \pi(A)\ar[r]^{\pi(f)}  & \pi(B)\ar[r]  & \pi(D)\ar[r]  & \pi(\Sigma(A))}
$$
is called a {\em standard triangle} in $\underline{\mcA}$. The class of triangles in $\underline{\mcA}$ consists of all triangles that are isomorphic to a standard triangle, for more details, we refer to \cite{Ha}.

Let $B$ be an object in a Frobenius category $(\mcA;\mcE)$, and fix a conflation $\beta\colon  B\to E\to\Sigma(B)$ in $(\mcA;\mcE)$ with $E$ an injective object. Given a morphism $g\colon A\to\Sigma(B)$ in $\mcA$, the pullback of $\beta$ along $g$ gives rise to a commutative diagram of conflations:
$$\xymatrix@R=30pt@C=30pt{
    \gamma:~~~~B \ar[r]\ar@{=}@<2ex>[d]  & C \ar[r]\ar[d]  & A\ar[d]^{g}\\
    \beta:~~~~B \ar[r]                   & E \ar[r]        & \Sigma(B). }
$$
This operation $\epsilon\colon g\mapsto\gamma$ defines a morphism
$$\epsilon\colon \Hom(A, \Sigma(B))\to\Ext(A,B)$$
of abelian groups. Since $E$ is injective, every conflation $\gamma\in\Ext(A,B)$ arises as such a pullback of $\beta$, so that the morphism $\epsilon$ is onto. The kernel of $\epsilon$ is given by the subgroup of $\Hom(A, \Sigma(B))$ of those morphisms that factor through the now projective object $E$. This induces an isomorphism
$$\epsilon_{A,B}\colon \underline{\Hom}(\pi(A), \pi(\Sigma(B)))\cong\Ext(A,B),$$
natural in both variables. Naturality then implies that for every pair $(f,g)$ of morphisms in
$\mcA$, $\Ext(f,g)=0$ if and only if $\underline{\Hom}(\pi(f), \pi(\Sigma(g)))=0$, which yields the following observation.

\begin{prop}\label{P: bij in F}
Let $(\mcA;\mcE)$ be a Frobenius category. The rule $(\mcI, \mcJ)\mapsto(\pi(\mcI), \pi(\Sigma(\mcJ)))$ is a bijective correspondence between ideal cotorsion pairs in $(\mcA;\mcE)$ and ideal torsion pairs in the stable category $\underline{\mcA}$.
\end{prop}

The following lemma provides a sufficient condition for a precovering ideal in a Frobenius category to be $\Ext$-special, which can be used to verify Theorem \ref{T: bij com in F} below.

\begin{lem}\label{L: special in F}
Let $\mcI$ be a precovering ideal in a Frobenius category $(\mcA;\mcE)$. If every object $A\in\mcA$ has an $\mcI$-precover that is a deflation, then $\mcI$ is an $\Ext$-special precovering ideal.
\end{lem}
\begin{proof} Let $A$ be an object in $\mcA$. Suppose that there is a conflation
$$\xymatrix@R=30pt@C=30pt{
     \alpha:~~~~F \ar[r]    & T \ar[r]^{i}   & A, }
$$
where $i\colon T\to A$ is an $\mcI$-precover of $A$. Because $(\mcA;\mcE)$ has enough projective objects, the conflation $\alpha$ arises as the pushout of a conflation $\beta$ as shown
$$\xymatrix@R=30pt@C=30pt{
     \;\;\;\;\beta:\Omega(A)\ar[r]\ar@<2.2ex>[d]_{g} & P\ar[r]\ar[d]  & A\ar@{=}[d]\\
     \alpha:~~~~F\ar[r]                              & T\ar[r]^{i}    & A, }
$$
where $P$ is projective. It suffices to verify that the morphism $g\colon \Omega(A)\to F$ belongs to $\mcI^{\perp_{1}}$ or, equivalently, that
$$\underline{\Hom}(\pi(\mcI), \pi(\Sigma(g)))=0.$$
But this morphism of conflations gives rise to the triangle
$$\xymatrix@R=30pt@C=30pt{
   \pi(F)\ar[r] & \pi(T)\ar[r]^{\pi(i)} & \pi(A)\ar[r]^(.45){\pi(\Sigma(g))} & \pi(\Sigma(F))}
$$
in the stable category $\underline{\mcA}$. As in the proof of Proposition \ref{P: in F}, the morphism $\pi(i)\colon \pi(T)\to \pi(A)$ is a $\pi(\mcI)$-precover of $\pi(A)$, so that for every morphism $i'\colon T'\to A$ in $\mcI$, the composition $\pi(\Sigma(g))\pi(i')$ is zero. Thus $\pi(\Sigma(g))\in \ell(\pi(\mcI))$ and therefore $\Hom(\pi(\mcI), \pi(\Sigma(g)))=0$, by Theorem \ref{T: orth and ann}.
\end{proof}

Combining with Propositions \ref{P: in F} and \ref{P: bij in F}, we get the following result.

\begin{thm}\label{T: bij com in F}
Let $(\mcA; \mcE)$ be a Frobenius category. An ideal cotorsion pair $(\mcI, \mcJ)$ in $(\mcA;\mcE)$ is complete if and only if the corresponding ideal torsion pair $(\pi(\mcI), \pi(\Sigma(\mcJ)))$ in $\underline{\mcA}$ is complete.
\end{thm}

\subsection{Applications}

\begin{exa}\rm
Let $\Lambda$ be an artin algebra, and $\Lambda$-mod the abelian category of finitely generated left $\Lambda$-modules. Then every object $M$ in $\Lambda$-mod has a decomposition $M=\oplus_{i=1}^{m}M_{i}$, where each $M_{i}$ is indecomposable with a local endomorphism ring. The {\em Jacobson radical} $\mcI=\Jac(\Lambda$-mod) is the ideal of $\Lambda$-mod obtained by taking the intersection of all left (resp., right) maximal ideals of $\Lambda$-mod, which is determined by the abelian groups of morphisms of indecomposable modules. If $U$ and $V$ are both indecomposable modules in $\Lambda$-mod, but not isomorphic, then $\mcI(U,V)=\Hom(U,V)$, while $\mcI(U,U)=\Jac(\End(U))$, the Jacobson radical of the endomorphism ring of $U$. This characterization of $\mcI$ implies that if $V\in\Lambda$-mod is indecomposable, then a morphism $f\colon M\to V$ belongs to $\mcI(M, V)$ if and only if it is not a split epimorphism.

By the work of Auslander and Reiten \cite{AR}, every indecomposable module $V$ in $\Lambda$-mod has an $\mcI$-precover $c\colon C\to V$, which is not a split epimorphism with the property that every morphism $c'\colon C'\to V$ that is not a split epimorphism factors through $c$, as in
$$\xymatrix@R=30pt@C=30pt{
                & C' \ar[d]^{c'}\ar@{-->}[ld] \\
  C \ar[r]^{c}  & V. }
$$
By Proposition \ref{P: precover complete}, the ideal torsion pair $(r(\ell(\mcI)), \ell(\mcI))$ generated by $\mcI$ is complete in $\Lambda$-mod. If an indecomposable module $V\in\Lambda$-mod is not projective, then the $\mcI$-precover $c\colon C\to V$ of $V$ is an epimorphism. This implies that the $\ell(\mcI)$-envelope, given by the cokernel of $c$, is the zero morphism $0\colon V\to 0$, and that the $r(\ell(\mcI))$-cover, given by the kernel of this zero morphism, is the identity morphism $1_{V}\colon V\to V$. If the indecomposable module $V=P$ is projective in $\Lambda$-mod, then it has a unique maximal submodule $J(P)$ and the $\mcI$-precover of $P$ is given by the inclusion morphism $c_{P} \colon J(P)\to P$. The $\ell(\mcI)$-envelope is given by the cokernel of $c_{P}$, which is the quotient morphism $P\to top(P)=P/J(P)$, and the $r(\ell(\mcI))$-cover is given by $c_{P}$ itself.

If the artin algebra $\Lambda$ is Quasi-Frobenius, then $\Lambda$-$\proj=\Lambda$-$\inj$. It follows from Lemma \ref{L: special in F} that for every almost split (Auslander-Reiten) sequence
$$\xymatrix@R=30pt@C=30pt{
      0 \ar[r]  & U \ar[r]^{f}  & W \ar[r]^{g}  & V \ar[r] & 0,}
$$
the morphism $g\colon W\to V$ is an $\Ext$-special $\mcI$-precover of $V$. Therefore, every almost split sequence in $\Lambda$-mod arises as a pushout along a morphism in $\mcI^{\perp_{1}}$, the {\em Auslander-Reiten cophantom ideal} \cite{FGHT}. This provides a new proof of the dual result of \cite[Corollary 51]{FGHT} that if $\Lambda$ is a Quasi-Frobenius artin algebra, then every almost split sequence in $\Lambda$-mod is right special.
\end{exa}

We conclude this paper with the following result, which gives a positive answer to Question 29 of \cite{FGHT} for complete ideal cotorsion pairs in a Frobenius category.

\begin{coro}\label{C: ob cor com}
Let $(\mcI, \mcJ)$ be a complete ideal cotorsion pair in a Frobenius category $(\mcA;\mcE)$. If both $\mcI$ and $\mcJ$ are object ideals, then the corresponding cotorsion pair $(\Ob(\mcI), \Ob(\mcJ))$ is complete.
\end{coro}

\begin{proof}
By Theorem \ref{T: bij com in F}, the corresponding ideal torsion pair $(\pi(\mcI), \pi(\Sigma(\mcJ)))$ is complete in the stable category $\underline{\mcA}$. Since both $\mcI$ and $\mcJ$ are object ideals, $\pi(\mcI)$ and $\pi(\Sigma(\mcJ))$ are also object ideals. It follows that the pair $(\Ob(\pi(\mcI)), \Ob(\pi(\Sigma(\mcJ))))$  in $\underline{\mcA}$ is a torsion theory in the sense of \cite[Definition 2.2]{IY}. According to  \cite[Proposition 3.3]{SS}, the corresponding cotorsion pair $(\Ob(\mcI), \Ob(\mcJ))$ in $(\mcA;\mcE)$ is complete.
\end{proof}

\bigskip



\begin{thebibliography}{99}

\bibitem{AS0} J. Asadollahi and S. Sadeghi, Higher ideal approximation theory, {\em Trans. Amer. Math. Soc.} {\bf 375}(3) (2022), 2113-2145.
\bibitem{AR} M. Auslander and I. Reiten, Representation theory of artin algebras \uppercase\expandafter {\romannumeral3}: Almost split sequences, {\em Comm. Algebra} {\bf 3}(3) (1975), 239-294.
\bibitem{AS} M. Auslander and S. Smal\o, Preprojective modules over artin algebras, {\em J. Algebra} {\bf 66}(1) (1980), 61-122.
\bibitem{BG} D.J. Benson and G.Ph. Gnacadja, Phantom maps and purity in modular representation theory \uppercase\expandafter {\romannumeral1}, \emph{Fund. Math.} {\bf 161}(1-2) (1999), 37-91.
\bibitem{BM} S. Breaz and G.C. Modoi, Ideal cotorsion theories in triangulated categories, {\em J. Algebra} {\bf 567} (2021), 475-532.
\bibitem{B} T. B\"{u}hler, Exact categories, {\em Expo. Math.} {\bf 28}(1) (2010), 1-69.
\bibitem{C} J.D. Christensen, Ideals in triangulated categories: phantoms, ghosts and skeleta, {\em Adv. Math.} {\bf 136}(2) (1998), 284-339.
\bibitem{E} E.E. Enochs, Injective and flat covers, envelopes and resolvents, {\em Israel J. Math.} {\bf 39}(3) (1981), 189-209.
\bibitem{EJ} E.E. Enochs and O.M.G. Jenda, {\em Relative Homological Algebra,} De Gruyter Expositions in Mathematics {\bf 30}, W. de Gruyter, Berlin, New York, 2000.
\bibitem{FGHT} X.H. Fu, P.A. Guil Asensio, I. Herzog and B. Torrecillas, Ideal approximation theory, {\em Adv. Math.} {\bf 244} (2013), 750-790.
\bibitem{FH} X.H. Fu and I. Herzog, Powers of the phantom ideal, {\em Proc. Lond. Math. Soc. (3)} {\bf 112}(4) (2016), 714-752.
\bibitem{FHHZ} X.H. Fu, I. Herzog, J.S. Hu and H.Y. Zhu, Lattice theoretic properties of approximating ideals, {\em J. Pure Appl. Algebra} {\bf 226}(7) (2022), 106986.
\bibitem{GT} R. G\"{o}bel and J. Trlifaj, {\em Approximations and Endomorphism Algebras of Modules,} De Gruyter Expositions in Mathematics {\bf 41}, W. de Gruyter, Berlin, New York, 2006.
\bibitem{Ha} D. Happel, {\em Triangulated Categories in the Representation Theory of Finite Dimensional Algebras}, London Math. Soc. Lecture Note Ser. {\bf 119}, Cambridge Univ. Press, Cambridge, 1988.
\bibitem{IY} O. Iyama and Y. Yoshino, Mutation in triangulated categories and rigid Cohen-Macaulay modules, {\em Invent. Math.} {\bf 172}(1) (2008), 117-168.
\bibitem{NP} H. Nakaoka, Y. Palu, Extriangulated categories, Hovey twin cotorsion pairs and model structures, {\em Cah. Topol. G\'{e}om. Diff\'{e}r. Cat\'{e}g}. {\bf 60}(2) (2019), 117-193.
\bibitem{SL} L. Salce, Cotorsion theories for abelian groups, {\em Symposia Mathematica} \uppercase\expandafter {\romannumeral23} (1979), 11-32.
\bibitem{SS} M. Saor\'{i}n and J. \v{S}\v{t}ov\'{i}\v{c}ek, On exact categories and applications to triangulated adjoints and model structures, {\em Adv. Math.} {\bf 228}(2) (2011), 968-1007.
\bibitem{Sch} K. Schlegel, Ideal Torsion Pairs for Artin Algebras, preprint, May 14, 2024.
\bibitem{S} B. Stenstr\"{o}m, {\em Rings of Quotients: An Introduction to Methods of Ring Theory}, Grundlehren Math. Wiss. {\bf 217}, Springer-Verlag, Berlin, 1975.


\end{thebibliography}
\end{document}